\documentclass{amsart}%
\usepackage{graphicx}
\usepackage{amsmath}
\usepackage{amsfonts}
\usepackage{amssymb}
\usepackage{ifpdf}%
\providecommand{\U}[1]{\protect\rule{.1in}{.1in}}
\newtheorem{theorem}{Theorem}

\newtheorem{definition}[theorem]{Definition}

\newtheorem{lemma}[theorem]{Lemma}

\newtheorem{proposition}[theorem]{Proposition}
\newtheorem{remark}[theorem]{Remark}

\makeatletter \@namedef{subjclassname@2020}{\textup{2020}
Mathematics Subject Classification} \makeatother
\begin{document}
\title[Oscillation Functionals and Embeddings in r.i. spaces]{Oscillation Functionals and Embeddings in Rearrangement-Invariant Spaces}
\author{Joaquim Mart\'{\i}n}
\address{Department of Mathematics\\
Universitat Aut\`onoma de Barcelona} \email{Joaquin.Martin@uab.cat}
\thanks{Partially supported by Grants PID2024-160507NB-I00 and PID2024-155917NB-I00
funded by MCIN/AEI/10.13039/501100011033.} \subjclass[2020]{46E30,
46E35, 46B70} \keywords{Oscillation functional,
rearrangement-invariant spaces, Boyd indices, Hardy-type operators,
Sobolev embeddings}

\begin{abstract}
We study embeddings associated with oscillation functionals in
rearrangement-invariant spaces. More precisely, given a positive
function \(\psi\), we analyze how the interaction between the
geometry of the underlying space and the growth of \(\psi\)
determines the behaviour of these embeddings, leading to a natural
classification into subcritical, supercritical and critical regimes.

We prove that in the critical regime logarithmic refinements of
Hansson type appear, governed by a deviation function associated
with the quotient \(\psi/\varphi_X\), where \(\varphi_X\) is the
fundamental function of the underlying space. This leads to explicit
Hansson-type targets and, in the bounded case of the deviation
function, to Trudinger-type consequences. The results recover and
extend several classical endpoint embeddings.
\end{abstract}

\maketitle

\section{Introduction}

The oscillation functional
\[
O(f,t)=f^{**}(t)-f^{*}(t), \qquad0<t<1,
\]
plays an important role in many problems of analysis. Here $f^{*}$
and $f^{**}$ denote the decreasing rearrangement of $f$ and its
maximal average, respectively (precise definitions, notation and
background material concerning the notions appearing in this
introduction and used throughout the paper are collected in
Section~\ref{sec:background}). The quantity $O(f,t)$ measures the
gap between the average size of the largest values of $f$ on a set
of measure $t$ and the boundary level $f^{*}(t)$.

Oscillation functionals of this type arise naturally in connection
with Sobolev and Besov-type embeddings, interpolation theory,
rearrangement inequalities, and related questions. Their systematic
use in the study of endpoint embeddings and symmetrization methods
was developed in work of M. Milman and collaborators (see, e.g.,
\cite{BMR,MP,MMP,MM3,MM6,MMas}). In many such situations one
encounters inequalities of the form
\begin{equation}
O(f,t)\le\psi(t)\,A(f,t), \label{EMBDOS_intro}%
\end{equation}
where $\psi$ reflects geometric or analytic features of the
underlying space such as isoperimetric profiles, volume growth, or
capacity estimates, while the functional $A(f,t)$ captures analytic
properties of $f$. Typical examples include gradients and fractional
derivatives
\cite{tal,tal1,maz'yabook,BMR,MMP,MM3,MM6,Cianchi1,CianchiPick,Kol1,MoPosi},
Besov-type embeddings \cite{CK,EEK,FMP}, Haj\l asz gradients and
moduli of continuity in metric measure spaces
\cite{Mas,Mo1,Mo2,Mo3,Mo4}, and interpolation functionals and sharp
maximal functions \cite{Grigoryan,Ler,SS,BK,BdVS,Ler0}. The
literature on these topics is extensive.

Given a rearrangement-invariant (r.i.) space $X$, the estimate
(\ref{EMBDOS_intro}) implies
\[
\left\|  \frac{O(f,t)}{\psi(t)}\right\|  _{X}
\leq\|A(f,\cdot)\|_{X},
\]
showing that oscillation functionals play a central role in many
embedding problems, including Sobolev, fractional Sobolev,
Besov-type and Haj\l asz-type embeddings, as well as interpolation
theory.

\medskip

Motivated by this framework, for an r.i. space $X$, an admissible
function $\psi$, and $0<r\le1$, we consider the oscillation space
\[
LS_{r}(X,\psi)=\left\{  f\in L^{0} :
\frac{O(|f|^{r},t)^{1/r}}{\psi(t)}\in X \right\}  .
\]

Although $LS_{r}(X,\psi)$ is itself rearrangement-invariant, it is
in general neither linear nor a lattice, and its defining functional
is not equivalent to a norm (see, e.g., \cite{CKMP,CGMP,KP,Ku}).
This reflects the nonlinear nature of the oscillation operator
$f\mapsto O(|f|^{r},t)^{1/r}$ and makes the analysis of
$LS_{r}(X,\psi)$ substantially more delicate than in the classical
r.i. setting. The role of the parameter $r$ becomes especially
transparent in the quasi-Banach setting; see
Section~\ref{sec:quasi}.

Our aim in this paper is to characterize those r.i.\ spaces $Y$ for
which
\begin{equation}
LS_{r}(X,\psi)\hookrightarrow Y. \label{main_embedding_intro}%
\end{equation}
Such estimates describe the gain of integrability produced by
control of the oscillation functional.

\medskip

The embedding (\ref{main_embedding_intro}) is governed by the
interaction between the growth of $\psi$ and the geometry of $X$,
encoded by the quotient
\begin{equation}
\frac{\psi(t)}{\varphi_{X}(t)},\qquad0<t<1, \label{quotient}%
\end{equation}
where $\varphi_{X}$ denotes the fundamental function of $X$. This
leads naturally to three qualitatively different regimes.

\medskip

\begin{itemize}
\item \emph{Supercritical regime.} The quotient (\ref{quotient}) is
sufficiently large, and the oscillation condition forces essentially
$L^{\infty}$-type behaviour (see Subsection~\ref{superc}).

\item \emph{Subcritical regime.} The quotient (\ref{quotient}) is dominated by
the geometry of $X$, and the embedding reduces to a maximal-type
description. In the classical $L^{p}$ setting, this recovers
Lorentz-type targets (see Subsection~\ref{sec:subcritical}).

\item \emph{Critical regime.} This is the borderline situation, where the
quotient (\ref{quotient}) no longer yields a purely power-type
description and logarithmic corrections appear (see
Subsection~\ref{sec:critical}).
\end{itemize}

The critical regime is the most delicate one, since the quotient
(\ref{quotient}) no longer has a simple power-type behaviour and
additional information is needed. In this case, (\ref{quotient}) is
replaced by the deviation function
\[
M(t):=\sup_{t<u<1}\frac{\psi(u)}{\varphi_{X}(u)},\qquad0<t<1.
\]
The asymptotic behaviour of \(M\) determines the endpoint form of
the embedding. When \(M\) is bounded, the critical estimate leads to
Trudinger-type consequences. In the natural borderline case in which
\(M\) has no residual power contribution, namely
\[
\underline{\beta}_M=\overline{\beta}_M=0,
\]
where \(\underline{\beta}_M\) and \(\overline{\beta}_M\) denote the
lower and upper fundamental indices of \(M\), the corresponding
endpoint targets are of Hansson type; see
Theorems~\ref{thm:hansson_Lalpha} and \ref{thm:hansson_convexified}.
More precisely, for each r.i.\ space satisfying an \(\alpha\)-lower
estimate we obtain a canonical \(L^\alpha\)-based Hansson-type
target whose norm is of the form
\[
\left(  \int_{0}^{1} \left(  \frac{f^{**}(t)}
{(\log(e/t))^{1/r}M(t)} \right) ^{\alpha}\frac{dt}{t} \right)
^{1/\alpha}.
\]
This provides an explicit logarithmic endpoint improvement in the
critical regime. In addition, when the space \(X\) satisfies a
compatible upper estimate, we obtain a second Hansson-type target
adapted to the geometry of \(X\), obtained by a suitable
convexification of \(X\). When the upper estimate of \(X\) already
matches the lower estimate used in the argument, this target reduces
to the Hansson-type space built directly on \(X\).

A further advantage of these Hansson-type targets is that they are
given explicitly in terms of the fundamental function of the
underlying space and the deviation function \(M\). By contrast,
abstract characterizations of optimal targets in terms of Hardy-type
operators (see Section~\ref{sec:hardy}) are often difficult to use
in concrete situations, since each family of spaces typically
requires a separate analysis. Our construction provides a unified
description valid for a large class of r.i. spaces.

The paper is organized as follows. Section~\ref{sec:background}
collects basic material on rearrangements, rearrangement-invariant
spaces, Boyd indices and growth indices. Section~\ref{sec:hardy}
establishes the equivalence between oscillation inequalities and the
boundedness of the associated Hardy-type operators.

Section~\ref{sec:main} contains the main embedding results,
organized into the supercritical, subcritical and critical regimes.
In the critical case we obtain Trudinger-type consequences and
explicit Hansson-type targets, including a canonical target and a
convexified target adapted to the geometry of the initial space.
Finally, Section~\ref{sec:quasi} explains how the Banach theory
extends to the quasi-Banach setting by means of $r$-convexification.
The appendix contains several auxiliary proofs.

\section{Background}

\label{sec:background}

We briefly collect notation and standard facts concerning
rearrangement-invariant (r.i.) spaces. Unless otherwise stated, all
function spaces are considered over the interval \((0,1)\) endowed
with Lebesgue measure. For further background we refer to
\cite{BS,LT2,MN,KPS,Boyd,BGT}. The material presented here provides
the structural framework for the analysis of the oscillation
inequalities and Hardy-type operators studied in the sequel.

Throughout the paper, \(A \preceq B\) means \(A \le C B\) for some
constant \(C>0\) independent of the relevant functions. We write
\(A\simeq B\) if both \(A\preceq B\) and \(B\preceq A\) hold. We
also say that a function \(f\) is almost increasing (almost
decreasing) if it is equivalent to an increasing (decreasing)
function.

\subsection{Rearrangements and r.i.\ spaces}

Let $(0,1)$ be endowed with Lebesgue measure. We denote by
$L^{0}(0,1)$ the space of all measurable functions on $(0,1)$ which
are finite almost everywhere and identified up to equality almost
everywhere. For $f\in L^{0}(0,1)$, its decreasing rearrangement is
\[
f^{\ast}(s)=\inf\{t>0:|\{x\in(0,1):|f(x)|>t\}|\leq s\},\quad
s\in(0,1).
\]
Associated to $f^{\ast}$, we consider the maximal function
\[
f^{\ast\ast}(t)=\frac{1}{t}\int_{0}^{t}f^{\ast}(s)\,ds,\qquad0<t<1,
\]
and, for $0<r\leq1$, the $r$-oscillation of $f^{\ast}$ defined by
\[
O(|f|^{r},t):=(|f|^{r})^{\ast\ast}(t)-(|f|^{r})^{\ast}(t).
\]
Notice that
\begin{equation}
\frac{d}{dt}(|f|^{r})^{\ast\ast}(t)=-\frac{O(|f|^{r},t)}{t}. \label{der2est}%
\end{equation}

A Banach function space $X$ is called rearrangement-invariant (r.i.)
if $\Vert f\Vert_{X}=\Vert g\Vert_{X}$ whenever $f^{\ast}=g^{\ast}$,
and if $|f|\leq|g|$ implies $\Vert f\Vert_{X}\leq\Vert g\Vert_{X}$.

The associate space $X^{\prime}$ of $X$ is defined by
\[
\|g\|_{X^{\prime}}=\sup_{\|f\|_{X}\le1}\int_{0}^{1} |f(s) g(s)|\,ds.
\]
It is also an r.i. space, and in fact the associate norm can be
obtained using only decreasing functions, namely
\begin{equation}
\|g\|_{X^{\prime}}=\sup_{\|f\|_{X}\le1}\int_{0}^{1}
f^{*}(t)g^{*}(t)\,dt.
\label{HolRea}%
\end{equation}
Furthermore, the following H\"{o}lder-type inequality holds
\begin{equation}
\int_{0}^{1} |f(s)g(s)|\,ds\le\|f\|_{X}\|g\|_{X^{\prime}}. \label{Holder}%
\end{equation}

A useful tool in the study of an r.i. space $X$ is the fundamental
function of $X$ defined by
\[
\varphi_{X}(t)=\|\chi_{(0,t)}\|_{X},\qquad0<t<1.
\]
This function is increasing with $\varphi_{X}(0^{+})=0$. For
example, if $X=L^{p}(0,1)$, then $\varphi_{L^{p}}(t)=t^{1/p}$. The
function $\varphi_{X}$ is quasi-concave and satisfies the duality
relation
\begin{equation}
\varphi_{X}(t)\varphi_{X^{\prime}}(t)=t. \label{funddual}%
\end{equation}

Let $p>0$ and let $X$ be an r.i. space; the $p$-convexification
$X^{(p)}$ of $X$ (see \cite{LT2,JS}) is defined by
\[
X^{(p)}=\{f:|f|^{p}\in X\},\qquad \Vert f\Vert_{X^{(p)}}=\Vert
|f|^{p}\Vert _{X}^{1/p}.
\]
It follows that
\[
\varphi_{X^{(p)}}(t)=\left(\varphi_{X}(t)\right)^{1/p}.
\]
If $p\geq1$, then $X^{(p)}$ is again an r.i. space.

We say that $X$ satisfies an upper (resp. lower) $\alpha$-estimate
if there exists a constant $C_{\alpha}>0$ such that for every finite
family of functions with pairwise disjoint supports
$\{f_{i}\}_{i=1}^{n}$ one has
\begin{align*}
\left\Vert \left(  \sum_{i=1}^{n}|f_{i}|^{\alpha}\right)  ^{1/\alpha
}\right\Vert _{X}  &  \leq C_{\alpha}\left(  \sum_{i=1}^{n}\Vert f_{i}%
\Vert_{X}^{\alpha}\right)  ^{1/\alpha},\\
\left(  \sum_{i=1}^{n}\Vert f_{i}\Vert_{X}^{\alpha}\right)
^{1/\alpha}  & \leq C_{\alpha}\left\Vert \left(
\sum_{i=1}^{n}|f_{i}|^{\alpha}\right) ^{1/\alpha}\right\Vert _{X}.
\end{align*}

\begin{remark}
\label{rem:convexification_estimates} If \(X\) satisfies an
\(\alpha\)-lower estimate, then \(X^{(p)}\) satisfies an \(\alpha
p\)-lower estimate. Similarly, if \(X\) satisfies a \(\rho\)-upper
estimate, then \(X^{(p)}\) satisfies a \(\rho p\)-upper estimate.
Indeed, for pairwise disjoint functions \(g_i\),
\[
\left\|\sum_i g_i\right\|_{X^{(p)}}^p = \left\|\sum_i
|g_i|^p\right\|_X,
\]
and the conclusion follows directly from the corresponding estimate
in \(X\).
\end{remark}

\subsection{The fundamental indices}

Let $\mathcal{A}$ be the class of positive functions
$\psi:(0,1)\to(0,\infty)$ such that
\begin{equation*}
m_{\psi}(t)=\sup_{\substack{0<s<1\\st<1}}\frac{\psi(st)}{\psi(s)}<\infty,
\qquad t>0. 
\end{equation*}
The function $m_{\psi}(t)$ is submultiplicative. Hence, by the
standard theory of submultiplicative functions, the following limits
exist (possibly infinite), and moreover coincide with the
corresponding supremum and infimum:
\begin{equation}
\underline{\beta}_{\psi}=\lim_{t\rightarrow0^{+}}\frac{\log
m_{\psi}(t)}{\log t}=\sup_{0<t<1}\frac{\log m_{\psi}(t)}{\log
t},\;\overline{\beta}_{\psi}
=\inf_{t>1}\frac{\log m_{\psi}(t)}{\log t} =\lim_{t\rightarrow\infty}%
\frac{\log m_{\psi}(t)}{\log t}. \label{fundind}%
\end{equation}
It is well known that if $\psi$ is increasing, then
\[
0\leq\underline{\beta}_{\psi}\leq\overline{\beta}_{\psi}\leq\infty.
\]

We denote by $\mathcal{A}_{0}$ the subclass of increasing functions
$\psi \in\mathcal{A}$ such that $\psi(0^{+})=0$ and
\[
0<\underline{\beta}_{\psi}\le\overline{\beta}_{\psi}<1.
\]

Given an r.i.\ space $X$ on $(0,1)$, the Zippin indices (see
\cite{Zipp}) of $X$ are defined as the fundamental indices of its
fundamental function $\varphi_{X}$.

\subsection{Boyd indices}

These indices were introduced by D.\thinspace W.~Boyd \cite{Boyd}
and govern the boundedness of Hardy-type operators and related
embeddings.

Let $X$ be an r.i. space on $(0,1)$. For $s>0$ we define the
dilation operator by
\[
(E_{s}f)(t)=
\begin{cases}
f(t/s), & 0<t<\min\{1,s\},\\
0, & \min\{1,s\}\le t<1.
\end{cases}
\]

The dilation function of $X$ is defined by
\[
h_{X}(s)=\|E_{s}\|_{X\to X}, \qquad s>0 .
\]

The lower and upper Boyd indices of $X$ are defined by
\begin{equation*}
\underline{\alpha}_{X}= \lim_{s\to0^{+}} \frac{\log h_{X}(s)}{\log
s}, \qquad\overline{\alpha}_{X}= \lim_{s\to\infty} \frac{\log
h_{X}(s)}{\log s}.
\end{equation*}
It follows from the definition of the Boyd indices that, for every
\(\varepsilon>0\), there exists a constant \(C_{\varepsilon}>0\)
such that
\begin{equation}
h_X(s)\le C_{\varepsilon}s^{\underline{\alpha}_X-\varepsilon},
\qquad 0<s<1, \label{eq:boyd_lower_estimate_h}
\end{equation}
and
\begin{equation*}
h_X(s)\le C_{\varepsilon}s^{\overline{\alpha}_X+\varepsilon}, \qquad
s>1. 
\end{equation*}

These indices satisfy
\[
0\leq\underline{\alpha}_{X}\leq\overline{\alpha}_{X}\leq1.
\]
Moreover,
\begin{equation}
\underline{\alpha}_{X^{\prime}}=1-\overline{\alpha}_{X},\qquad
\overline{\alpha}_{X^{\prime}}=1-\underline{\alpha}_{X}, \qquad
\underline{\alpha}_{X^{(p)}}=\frac{\underline{\alpha}_{X}}{p},
\qquad \overline{\alpha}_{X^{(p)}}=\frac{\overline{\alpha}_{X}}{p}.
\label{eq:boyd_dual_convexification}
\end{equation}
The relation between the Boyd indices and the fundamental indices of
$X$ is
\[
0\leq\underline{\alpha}_{X}\leq\underline{\beta}_{\varphi_{X}}\leq
\overline{\beta}_{\varphi_{X}}\leq\overline{\alpha}_{X}\leq1.
\]


\section{Embeddings and Hardy-type operators}

\label{sec:hardy}

To prove the main theorem of this section, we shall need the
following technical result. In the special case $r=1$ and $\psi$ a
power function, this estimate was proved in \cite{MP}. We show that
the same conclusion remains valid for general
$\psi\in\mathcal{A}_{0}$ and $0<r\leq1$. For completeness, we
include the proofs in the Appendix.

\begin{lemma}
\label{Mpus} For $0<r\leq1,$ consider the operators $P_{r}$ defined
by
\[
P_{r}f(t)=\left(  \frac{1}{t}\int_{0}^{t}|f(s)|^{r}\,ds\right)
^{1/r},\qquad t>0.
\]
Let $\psi\in\mathcal{A}_{0}$ and let $Y$ be an r.i. space. Then
there exists a constant $C_{\psi}<\infty$ such that
\[
\left\Vert \frac{P_{r}f}{\psi}\right\Vert _{Y}\leq
C_{\psi}\,\left\Vert \frac{f}{\psi}\right\Vert _{Y},\qquad f\in
L^{0}.
\]

\end{lemma}

The embedding problem studied in this paper is related to the
Hardy-type operators
\[
\overline{Q}_{\psi,r}f(t)=\left(
\int_{t}^{1}(\psi(s)|f(s)|)^{r}\,\frac
{ds}{s}\right)  ^{1/r}%
\]
and
\[
\overline{T}_{\psi,r}h(t)=\int_{t}^{1}\psi(s)^{r}h(s)\,\frac{ds}{s}.
\]

\begin{theorem}
\label{inclu} Let $0<r\leq1$, let $X$ and $Y$ be r.i.\ spaces and
let $\psi \in\mathcal{A}_{0}$. The following are equivalent:

\begin{enumerate}
\item There exists a constant $C>0$ such that, for every measurable $f$,
\begin{equation}
\Vert f\Vert_{Y}\leq C\left(  \left\Vert \frac{O(|f|^{r},\cdot)^{1/r}}%
{\psi(\cdot)}\right\Vert _{X}+\Vert f\Vert_{L^{r}}\right)  . \label{bbb}%
\end{equation}

\item There exists a constant $C>0$ such that, for every $h\in X^{(1/r)}$,
\[
\|\overline{T}_{\psi,r}h\|_{Y^{(1/r)}} \le C\|h\|_{X^{(1/r)}}.
\]

\item There exists a constant $C>0$ such that, for every $f\in X$,
\[
\|\overline{Q}_{\psi,r}f\|_{Y} \le C\|f\|_{X}.
\]

\end{enumerate}
\end{theorem}

\begin{proof}
$(ii)\Leftrightarrow(iii)$. Since
\[
\overline{Q}_{\psi,r}f(t)=\left(
\overline{T}_{\psi,r}(|f|^{r})(t)\right) ^{1/r},
\]
and, by the definition of convexification,
\[
\Vert\overline{Q}_{\psi,r}f\Vert_{Y}^{r} =\Vert\overline{T}_{\psi,r}%
(|f|^{r})\Vert_{Y^{(1/r)}},
\]
it follows that
\[
\overline{Q}_{\psi,r}:X\rightarrow Y\text{ is bounded} \quad
\Longleftrightarrow\quad\overline{T}_{\psi,r}:X^{(1/r)}\rightarrow
Y^{(1/r)}\text{ is bounded}.
\]

It remains to prove $(iii)\Leftrightarrow(i)$.

$(iii)\Rightarrow(i)$. Let $f\in L^{0}$ and assume that
\[
\left\Vert \frac{O(|f|^{r},\cdot)^{1/r}}{\psi(\cdot)}\right\Vert
_{X} +\|f\|_{L^{r}}<\infty.
\]
By (\ref{der2est}) and the Fundamental Theorem of Calculus, we
obtain
\begin{align*}
(|f|^{r})^{\ast\ast}(t)  &  = \int_{t}^{1}\left( (|f|^{r})^{\ast\ast
}(s)-(|f|^{r})^{\ast}(s)\right)  \frac{ds}{s} +(|f|^{r})^{\ast\ast}(1)\\
&  = \int_{t}^{1} O(|f|^{r},s)\frac{ds}{s} +(
|f|^{r})^{\ast\ast}(1).
\end{align*}
Hence
\begin{align*}
\left(  (|f|^{r})^{\ast\ast}(t)\right)  ^{1/r}  &  \preceq\left(
\int_{t}^{1} O(|f|^{r},s)\frac{ds}{s}\right)  ^{1/r} +\left(
(|f|^{r})^{\ast\ast
}(1)\right)  ^{1/r}\\
&  = \left(  \int_{t}^{1}\psi^{r}(s)\left(
\frac{O(|f|^{r},s)^{1/r}}{\psi (s)}\right)  ^{r}\frac{ds}{s}\right)
^{1/r} +\|f\|_{L^{r}}.
\end{align*}
Therefore,
\begin{align*}
\Vert f\Vert_{Y}  &  \leq\left\Vert \left(
(|f|^{r})^{\ast\ast}(t)\right)
^{1/r}\right\Vert _{Y}\\
&  \preceq\left\Vert \left(  \int_{t}^{1}\psi^{r}(s)\left(  \frac
{O(|f|^{r},s)^{1/r}}{\psi(s)}\right)  ^{r}\frac{ds}{s}\right)  ^{1/r}%
\right\Vert _{Y} +\|f\|_{L^{r}}\\
&  = \left\Vert \overline{Q}_{\psi,r}\left(  \frac{O(|f|^{r},\cdot)^{1/r}%
}{\psi(\cdot)}\right)  \right\Vert _{Y} +\|f\|_{L^{r}}\\
&  \preceq\left\Vert
\frac{O(|f|^{r},\cdot)^{1/r}}{\psi(\cdot)}\right\Vert _{X}
+\|f\|_{L^{r}}.
\end{align*}

$(i)\Rightarrow(iii)$. Let $f\in X$. Since $\overline{Q}_{\psi,r}f$
is positive and decreasing, $\left(  \overline{Q}_{\psi,r}f\right)
^{\ast}(t) =\overline{Q}_{\psi,r}f(t)$ and hence, by Fubini's
theorem,
\begin{align*}
O(\left(  \overline{Q}_{\psi,r}f\right)  ^{r},t)  &  =\frac{1}{t}\int_{0}%
^{t}\int_{s}^{1}(\psi(u)|f(u)|)^{r}\,\frac{du}{u}\,ds -\left(
\overline
{Q}_{\psi,r}f\right)  ^{r}(t)\\
&  =\frac{1}{t}\int_{0}^{t}(\psi(u)|f(u)|)^{r}\,du +\int_{t}^{1}%
(\psi(u)|f(u)|)^{r}\,\frac{du}{u} -\left(
\overline{Q}_{\psi,r}f\right)
^{r}(t)\\
&  =\frac{1}{t}\int_{0}^{t}(\psi(u)|f(u)|)^{r}\,du.
\end{align*}
Therefore,
\begin{align*}
\Vert\overline{Q}_{\psi,r}f\Vert_{Y}  &  \leq C\left\Vert
\frac{O(|\overline
{Q}_{\psi,r}f|^{r},\cdot)^{1/r}}{\psi(\cdot)}\right\Vert _{X}
+C\|\overline
{Q}_{\psi,r}f\|_{L^{r}} \qquad\text{(by (\ref{bbb}))}\\
&  = C\left\Vert \left(  \frac{1}{t\,\psi(t)^{r}}\int_{0}^{t}(\psi
(s)|f(s)|)^{r}\,ds\right)  ^{1/r}\right\Vert _{X}
+C\|\overline{Q}_{\psi
,r}f\|_{L^{r}}\\
&  = C\left\Vert \frac{P_{r}(\psi f)(t)}{\psi(t)}\right\Vert _{X}
+C\|\overline{Q}_{\psi,r}f\|_{L^{r}}\\
&  \preceq\|f\|_{X}+ \|\overline{Q}_{\psi,r}f\|_{L^{r}},
\end{align*}
where in the last step we used Lemma~\ref{Mpus}.

To estimate the second term, by Fubini's theorem,
\[
\Vert\overline{Q}_{\psi,r}f\Vert_{L^{r}}^{r}=\int_{0}^{1}\int_{t}^{1}%
(\psi(s)|f(s)|)^{r}\,\frac{ds}{s}\,dt=\int_{0}^{1}(\psi(s)|f(s)|)^{r}%
\,ds\leq\Vert\psi\Vert_{L^{\infty}}^{r}\Vert f\Vert_{L^{r}}^{r}.
\]
Using that every r.i.\ space on $(0,1)$ is continuously embedded
into $L^{1}$ and that $0<r\le1$, we get
\[
\|\overline{Q}_{\psi,r}f\|_{L^{r}}\preceq\|f\|_{L^{r}} \preceq\|f\|_{L^{1}%
}\preceq\|f\|_{X}.
\]
Combining the previous estimates, we conclude that
\[
\Vert\overline{Q}_{\psi,r}f\Vert_{Y}\preceq\Vert f\Vert_{X}.
\]
This proves that $\overline{Q}_{\psi,r}:X\rightarrow Y$ is bounded.
\end{proof}

\begin{proposition}
\label{CorOpti}

Let $0<r\leq1$. Let $X$ be an r.i.\ space and let $Z_{r}$ be the
r.i.\ space whose associate space $(Z_{r})^{\prime}$ is defined by
\begin{equation}
\Vert g\Vert_{(Z_{r})^{\prime}}:=\left\Vert
\psi(\cdot)^{r}\,g^{\ast\ast
}(\cdot)\right\Vert _{(X^{(1/r)})^{\prime}}. \label{Zrpsi}%
\end{equation}
Then $Z_{r}$ is the optimal r.i.\ range for the operator $\overline{T}%
_{\psi,r}$ in the sense that
\begin{equation}
\overline{T}_{\psi,r}:X^{(1/r)}\rightarrow Z_{r}\quad\text{is
bounded},
\label{B1}%
\end{equation}
and if $\overline{T}_{\psi,r}:X^{(1/r)}\rightarrow Y^{(1/r)}$ is
bounded for some r.i.\ space $Y$, then
\[
Z_{r}\hookrightarrow Y^{(1/r)}.
\]

Consequently, if we set $Z:=(Z_{r})^{(r)}$, then $Z$ is the optimal
r.i.\ range for $\overline Q_{\psi,r}$, that is,
\[
\overline Q_{\psi,r}:X\to Z \quad\text{is bounded},
\]
and if $\overline Q_{\psi,r}:X\to Y$ is bounded for some r.i.\ space
$Y$, then
\[
Z\hookrightarrow Y.
\]

\end{proposition}

\begin{proof}
First, we prove (\ref{B1}). We may assume that $f\in X^{(1/r)}$ and
$f\ge0$. From the definition of the associate norm (\ref{HolRea}),
and since $\overline{T}_{\psi,r}f$ is decreasing, we have
\begin{align*}
\Vert\overline{T}_{\psi,r}f\Vert_{Z_{r}}  &  =\sup_{\Vert g\Vert
_{(Z_{r})^{\prime}}\leq1}\int_{0}^{1}g(t)\, \overline{T}_{\psi,r}f(t)\,dt\\
&  =\sup_{\Vert
g\Vert_{(Z_{r})^{\prime}}\leq1}\int_{0}^{1}g^{\ast}(t)\left(
\int_{t}^{1}\psi(s)^{r}f(s)\,\frac{ds}{s}\right)  dt\\
&  =\sup_{\Vert
g\Vert_{(Z_{r})^{\prime}}\leq1}\int_{0}^{1}\psi(s)^{r}f(s) \left(
\frac{1}{s}\int_{0}^{s}g^{\ast}(u)\,du\right)  ds \qquad\text{(by
Fubini's theorem)}\\
&  =\sup_{\Vert g\Vert_{(Z_{r})^{\prime}}\leq1}\int_{0}^{1}
f(s)\,\psi
(s)^{r}g^{\ast\ast}(s)\,ds\\
&  \leq\sup_{\Vert g\Vert_{(Z_{r})^{\prime}}\leq1}\Vert
f\Vert_{X^{(1/r)}}
\,\Vert\psi(\cdot)^{r}g^{\ast\ast}(\cdot)\Vert_{(X^{(1/r)})^{\prime}}
\qquad\text{(by (\ref{Holder}))}\\
&  =\Vert f\Vert_{X^{(1/r)}}\sup_{\Vert
g\Vert_{(Z_{r})^{\prime}}\leq1} \Vert
g\Vert_{(Z_{r})^{\prime}}\\
&  =\Vert f\Vert_{X^{(1/r)}}.
\end{align*}

\medskip

\noindent\emph{Optimality.} Assume that $\overline{T}_{\psi,r}:X^{(1/r)}%
\rightarrow Y^{(1/r)}$ is bounded for some r.i.\ space $Y$. By
duality, the
adjoint operator $\overline{T}_{\psi,r}^{\ast}:(Y^{(1/r)})^{\prime}%
\to(X^{(1/r)})^{\prime}$ is bounded and
\[
\Vert\overline{T}_{\psi,r}^{\ast}g\Vert_{(X^{(1/r)})^{\prime}}
\preceq\Vert g\Vert_{(Y^{(1/r)})^{\prime}} \qquad\text{for all
}g\in(Y^{(1/r)})^{\prime}.
\]
Let $g\in(Y^{(1/r)})^{\prime}$ and let $h\in X^{(1/r)}$ be
nonnegative. By Fubini's theorem,
\begin{align*}
\int_{0}^{1}g(t)\,\overline{T}_{\psi,r}h(t)\,dt  &
=\int_{0}^{1}g(t)\left(
\int_{t}^{1}\psi(s)^{r}h(s)\,\frac{ds}{s}\right)  dt\\
&  =\int_{0}^{1}\psi(s)^{r}h(s)\left(  \frac{1}{s}\int_{0}^{s}
g(u)\,du\right)  ds.
\end{align*}
Taking the supremum over all $h$ with $\|h\|_{X^{(1/r)}}\le1$, we
obtain
\[
\left\Vert \psi(s)^{r}\left(  \frac{1}{s}\int_{0}^{s}g(u)\,du\right)
\right\Vert _{(X^{(1/r)})^{\prime}} =\Vert\overline{T}_{\psi,r}^{\ast}%
g\Vert_{(X^{(1/r)})^{\prime}} \preceq\Vert
g\Vert_{(Y^{(1/r)})^{\prime}}.
\]
In particular, replacing $g$ by $g^{\ast}$ we get
\[
\Vert\psi(\cdot)^{r}g^{\ast\ast}(\cdot)\Vert_{(X^{(1/r)})^{\prime}}%
\preceq\Vert g^{\ast}\Vert_{(Y^{(1/r)})^{\prime}} =\Vert g\Vert_{(Y^{(1/r)}%
)^{\prime}},
\]
which by (\ref{Zrpsi}) means precisely that
\[
\Vert g\Vert_{(Z_{r})^{\prime}}\preceq\Vert
g\Vert_{(Y^{(1/r)})^{\prime}} \qquad\text{for all
}g\in(Y^{(1/r)})^{\prime},
\]
that is,
\[
(Y^{(1/r)})^{\prime}\hookrightarrow(Z_{r})^{\prime}.
\]
Taking associate spaces, we conclude that
\[
Z_{r}\hookrightarrow Y^{(1/r)}.
\]
Hence $Z_{r}$ is the optimal r.i.\ range for
$\overline{T}_{\psi,r}$.

Finally, by Theorem~\ref{inclu}, if
$\overline{Q}_{\psi,r}:X\rightarrow Y$ is bounded for an r.i.\ space
$Y$, then
\[
\overline{T}_{\psi,r}:X^{(1/r)}\rightarrow Y^{(1/r)}
\]
is bounded. Applying the previous conclusion, we obtain
\[
Z_{r}\hookrightarrow Y^{(1/r)}.
\]
Since convexification preserves continuous embeddings, it follows
that
\[
Z=(Z_{r})^{(r)}\hookrightarrow(Y^{(1/r)})^{(r)}=Y,
\]
which proves that $Z$ is the optimal r.i.\ range for
$\overline{Q}_{\psi,r}$.
\end{proof}

\begin{remark}
The description of the space $Z$ provides a clean and theoretically
optimal formulation. However, its explicit identification is usually
difficult in practice, since it requires understanding associate
norms of the form
\[
g\mapsto\left\Vert \psi(\cdot)^{r}g^{\ast\ast}(\cdot)\right\Vert
_{(X^{(1/r)})^{\prime}},
\]
as introduced in Proposition~\ref{CorOpti}.
\end{remark}

\begin{remark}
The identification of optimal r.i. ranges for the operators
$\overline {T}_{\psi,r}$ and $\overline{Q}_{\psi,r}$ given in
Proposition~\ref{CorOpti} is closely related to the general theory
of boundedness of classical linear operators on r.i.\ spaces; see,
in particular, \cite{EMMP} and the references therein.
\end{remark}


\section{Three embedding regimes}

\label{sec:main}

In this section we analyze the embeddings associated with the
oscillation inequality according to the interaction between the
geometry of the space $X$ and the growth of the function $\psi$.

This leads to three qualitatively different regimes. In the
supercritical case, the oscillation inequality yields an
$L^{\infty}$-type embedding. In the subcritical case, it is
equivalent to a maximal-type description. In the critical case,
logarithmic corrections appear and give rise to Hansson-type
targets.

The Boyd indices of $X$ and the growth indices of $\psi$ will be the
main parameters in this analysis.

\subsection{The supercritical regime}

\label{superc}

\begin{theorem}
\label{supercrit} Let $0<r\leq1$, let $X$ be an r.i.\ space and let
$\psi \in\mathcal{A}_{0}$. Then the following statements are
equivalent:

\begin{enumerate}
\item There exists a constant $C>0$ such that
\begin{equation}
\Vert f\Vert_{L^{\infty}}\leq C\left(  \left\Vert
\frac{O(|f|^{r},\cdot )^{1/r}}{\psi(\cdot)}\right\Vert _{X} +\Vert
f\Vert_{L^{r}}\right)
\label{intoLoo}%
\end{equation}
for every measurable $f$.

\item
\begin{equation}
\left\Vert \frac{\psi(s)^{r}}{s}\chi_{(0,1)}(s)\right\Vert _{(X^{(1/r)}%
)^{\prime}}<\infty. \label{kernel_global}%
\end{equation}

\end{enumerate}
\end{theorem}

\begin{proof}
We first prove that (\ref{kernel_global}) implies (\ref{intoLoo}).
By (\ref{der2est}) and the Fundamental Theorem of Calculus,
\[
\|f\|_{L^{\infty}}^{r} =\lim_{t\to0^{+}} (|f|^{r})^{**}(t) =\int_{0}%
^{1}O(|f|^{r},s)\,\frac{ds}{s}+\left(  |f|^{r}\right)  ^{**}(1).
\]
Since
\[
\left(  |f|^{r}\right)  ^{**}(1)=\int_{0}^{1} (f^{*}(s))^{r}\,ds=\|f\|_{L^{r}%
}^{r},
\]
and $X^{(1/r)}$ is an r.i.\ space, H\"older's inequality yields
\begin{align*}
\|f\|_{L^{\infty}}^{r}  &
\le\int_{0}^{1}\frac{O(|f|^{r},s)}{\psi(s)^{r}}\,
\frac{\psi(s)^{r}}{s}\chi_{(0,1)}(s)\,ds+\|f\|_{L^{r}}^{r}\\[4pt]
&  \le\left\Vert \frac{O(|f|^{r},\cdot)}{\psi(\cdot)^{r}}\right\Vert
_{X^{(1/r)}} \left\Vert
\frac{\psi(s)^{r}}{s}\chi_{(0,1)}(s)\right\Vert
_{(X^{(1/r)})^{\prime}} +\|f\|_{L^{r}}^{r}\\
&  = \left\Vert
\frac{O(|f|^{r},\cdot)^{1/r}}{\psi(\cdot)}\right\Vert _{X}^{r}
\left\Vert \frac{\psi(s)^{r}}{s}\chi_{(0,1)}(s)\right\Vert _{(X^{(1/r)}%
)^{\prime}} +\|f\|_{L^{r}}^{r}.
\end{align*}
Hence (\ref{intoLoo}) follows.

We now prove the converse implication. Suppose that
(\ref{kernel_global}) fails, that is,
\[
\left\|  \frac{\psi(s)^{r}}{s}\chi_{(0,1)}(s) \right\|
_{(X^{(1/r)})^{\prime }} =\infty.
\]
Then, by the definition of the associate norm, there exists a
sequence $h_{n}\geq0$ with
\[
\Vert h_{n}\Vert_{X^{(1/r)}}\leq1
\]
such that
\[
\int_{0}^{1}h_{n}(s)\frac{\psi(s)^{r}}{s}\,ds\to\infty.
\]

Define
\[
g_{n}(t)=\int_{t}^{1}h_{n}(s)\frac{\psi(s)^{r}}{s}\,ds
\]
and let
\[
f_{n}=g_{n}^{1/r}.
\]
Then
\[
\Vert f_{n}\Vert_{L^{\infty}}^{r} =g_{n}(0)
=\int_{0}^{1}h_{n}(s)\frac {\psi(s)^{r}}{s}\,ds\to\infty.
\]

On the other hand, by Fubini's theorem,
\[
O(|f_{n}|^{r},t)=\frac{1}{t}\int_{0}^{t}\psi(s)^{r}h_{n}(s)\,ds.
\]
Therefore
\begin{align*}
\left(  \frac{O(|f_{n}|^{r},t)}{\psi(t)^{r}}\right)  ^{1/r}  &
=\left( \frac{1}{t\,\psi(t)^{r}}\int_{0}^{t}\left(
\psi(s)h_{n}(s)^{1/r}\right)
^{r}\,ds\right)  ^{1/r}\\
&  =\frac{P_{r}\left(  \psi\,h_{n}^{1/r}\right)  (t)}{\psi(t)}.
\end{align*}
Thus, by Lemma~\ref{Mpus},
\begin{align*}
\left\Vert \left(
\frac{O(|f_{n}|^{r},\cdot)}{\psi(\cdot)^{r}}\right)
^{1/r}\right\Vert _{X}  &  =\left\Vert \frac{P_{r}\left(  \psi\,h_{n}%
^{1/r}\right)  (\cdot)}{\psi(\cdot)}\right\Vert _{X}\\
&  \preceq\Vert h_{n}^{1/r}\Vert_{X}=\Vert h_{n}\Vert_{X^{(1/r)}}^{1/r}%
\preceq1.
\end{align*}

Finally,
\begin{align*}
\Vert f_{n}\Vert_{L^{r}}^{r}  &  = \int_{0}^{1}\left(  \int_{t}^{1}%
h_{n}(s)\frac{\psi(s)^{r}}{s}\,ds\right)  dt\\
&  = \int_{0}^{1}h_{n}(s)\psi(s)^{r}\,ds\\
&  \le\Vert h_{n}\Vert_{X^{(1/r)}}
\Vert\psi(s)^{r}\chi_{(0,1)}(s)\Vert
_{(X^{(1/r)})^{\prime}}\\
&  \le\Vert h_{n}\Vert_{X^{(1/r)}}\psi(1)^{r}\,\Vert\chi_{(0,1)}%
\Vert_{(X^{(1/r)})^{\prime}}\\
&  \preceq1.
\end{align*}

Thus the right-hand side of (\ref{intoLoo}) remains bounded, while
$\Vert f_{n}\Vert_{L^{\infty}}\to\infty$, which contradicts
(\ref{intoLoo}). Hence (\ref{kernel_global}) must hold.
\end{proof}

The following result provides a sufficient condition for the
embedding into $L^{\infty}$ in terms of the Boyd and growth indices.

\begin{proposition}
Let $0<r\leq1$, let $\psi\in\mathcal{A}_{0}$, and let $X$ be an
r.i.\ space. If
\[
\overline{\alpha}_{X}<\underline{\beta}_{\psi},
\]
then
\begin{equation}
\Vert f\Vert_{L^{\infty}}\preceq\left\Vert \frac{(O(|f|^{r},\cdot))^{1/r}%
}{\psi(\cdot)}\right\Vert _{X}+\Vert f\Vert_{L^{r}}. \label{dos}%
\end{equation}

\end{proposition}

\begin{proof}
By Theorem~\ref{supercrit}, it suffices to prove that
\begin{equation}
\left\Vert \frac{\psi(s)^{r}}{s}\chi_{(0,1)}(s)\right\Vert _{(X^{(1/r)}%
)^{\prime}}<\infty. \label{weight_general}%
\end{equation}

For $k\geq0$ let $I_{k}=[2^{-k-1},2^{-k})$. Since $\chi_{(0,1)}=\sum
_{k=0}^{\infty}\chi_{I_{k}}$, we have
\[
\left\Vert \frac{\psi(s)^{r}}{s}\chi_{(0,1)}(s)\right\Vert _{(X^{(1/r)}%
)^{\prime}} \leq\sum_{k=0}^{\infty}\left\Vert
\frac{\psi(s)^{r}}{s}\chi _{I_{k}}(s)\right\Vert
_{(X^{(1/r)})^{\prime}}.
\]

For $s\in I_{k}$ we have $2^{-k-1}\leq s\leq2^{-k}$ and thus
\[
\frac{\psi(s)^{r}}{s}\leq2^{k+1}\,\psi(2^{-k})^{r},
\]
so
\[
\left\Vert \frac{\psi(s)^{r}}{s}\chi_{I_{k}}\right\Vert
_{(X^{(1/r)})^{\prime
}} \leq2^{k+1}\,\psi(2^{-k})^{r}\Vert\chi_{I_{k}}\Vert_{(X^{(1/r)})^{\prime}%
}.
\]
Moreover, $\chi_{I_{k}}=E_{2^{-k}}\chi_{[1/2,1)}$, hence
\[
\Vert\chi_{I_{k}}\Vert_{(X^{(1/r)})^{\prime}} \leq h_{(X^{(1/r)})^{\prime}%
}(2^{-k})\,\Vert\chi_{[1/2,1)}\Vert_{(X^{(1/r)})^{\prime}}.
\]

Since $\overline{\alpha}_{X}<\underline{\beta}_{\psi}$, it follows
from (\ref{eq:boyd_dual_convexification}) that
\[
\underline{\alpha}_{(X^{(1/r)})^{\prime}} =
1-\overline{\alpha}_{X^{(1/r)}} = 1-r\,\overline{\alpha}_{X}
>
1-r\,\underline{\beta}_{\psi}.
\]
Choose $\delta$ such that
\begin{equation}
1-r\,\underline{\beta}_{\psi} < \delta <
\underline{\alpha}_{(X^{(1/r)})^{\prime}}. \label{betar}
\end{equation}
Then, by (\ref{eq:boyd_lower_estimate_h}), there exists $c>0$ such
that
\begin{equation}
h_{(X^{(1/r)})^{\prime}}(2^{-k}) \leq c\,2^{-k\delta}, \qquad
k\geq0. \label{hdecay}
\end{equation}

On the other hand, since $\psi\in\mathcal{A}_{0}$, it follows from
the definition of the fundamental indices \((\ref{fundind})\) that,
for every $0<\varepsilon<\underline{\beta}_{\psi}$, there exists
$C_{\varepsilon}>0$ such that
\[
\psi(2^{-k}) \leq
C_{\varepsilon}\,2^{-k(\underline{\beta}_{\psi}-\varepsilon)},
\qquad k\geq0.
\]

Combining this with (\ref{hdecay}), we get
\[
\left\Vert \frac{\psi(s)^{r}}{s}\chi_{(0,1)}(s) \right\Vert
_{(X^{(1/r)})^{\prime}} \preceq \sum_{k=0}^{\infty}
2^{k}\psi(2^{-k})^{r}2^{-k\delta} \leq \sum_{k=0}^{\infty}
2^{k(1-r\underline{\beta}_{\psi}+r\varepsilon-\delta)}.
\]

By (\ref{betar}) we can choose $\varepsilon>0$ so that
\[
1-r\underline{\beta}_{\psi}+r\varepsilon-\delta<0,
\]
and therefore the series converges. Hence (\ref{weight_general})
holds, and (\ref{dos}) follows from
Theorem~\ref{supercrit}.\end{proof}

\subsection{The subcritical regime}

\label{sec:subcritical}

\begin{theorem}
Let $0<r\leq1$, let $\psi\in\mathcal{A}_{0}$, and let $X$ be an
r.i.\ space. Assume that
\[
\underline{\alpha}_{X}>\overline{\beta}_{\psi}.
\]
Then
\[
\left\Vert \frac{O(|f|^{r},\cdot)^{1/r}}{\psi(\cdot)}\right\Vert
_{X}+\Vert
f\Vert_{L^{r}}\simeq\left\Vert \frac{((|f|^{r})^{\ast\ast}(\cdot))^{1/r}}%
{\psi(\cdot)}\right\Vert _{X}\simeq\left\Vert \frac{f^{\ast\ast}(\cdot)}%
{\psi(\cdot)}\right\Vert _{X}.
\]

In particular, the resulting space is independent of $r$.
\end{theorem}

\begin{proof}
Let $Y$ be the r.i.\ space defined by
\[
\Vert f\Vert_{Y}:=\left\Vert \frac{\left(
(|f|^{r})^{\ast\ast}(t)\right) ^{1/r}}{\psi(t)}\right\Vert _{X}.
\]
Clearly,
\[
\left\Vert \frac{O(|f|^{r},\cdot)^{1/r}}{\psi(\cdot)}\right\Vert _{X}%
\leq\left\Vert \frac{\left(  (|f|^{r})^{\ast\ast}(\cdot)\right)  ^{1/r}}%
{\psi(\cdot)}\right\Vert _{X}.
\]
Also, since $\psi$ is increasing on $(0,1)$, for every $0<t<1$ we
have
\[
\frac{\left(  (|f|^{r})^{\ast\ast}(t)\right)
^{1/r}}{\psi(t)}\geq\frac{\Vert f\Vert_{L^{r}}}{\psi(1)},
\]
and therefore
\[
\Vert f\Vert_{L^{r}}\preceq\left\Vert \frac{\left(  (|f|^{r})^{\ast\ast}%
(\cdot)\right)  ^{1/r}}{\psi(\cdot)}\right\Vert _{X}.
\]
Thus
\begin{equation}
\left\Vert \frac{O(|f|^{r},\cdot)^{1/r}}{\psi(\cdot)}\right\Vert
_{X}+\Vert
f\Vert_{L^{r}}\preceq\left\Vert \frac{\left(  (|f|^{r})^{\ast\ast}%
(\cdot)\right)  ^{1/r}}{\psi(\cdot)}\right\Vert _{X}. \label{localsub1}%
\end{equation}

To prove the converse inequality, by Theorem~\ref{inclu} it suffices
to show that $\overline{Q}_{\psi,r}$ is bounded from $X$ to $Y$. We
have
\begin{align*}
\Vert\overline{Q}_{\psi,r}f\Vert_{Y}  &  = \left\Vert
\frac{1}{\psi(t)} \left[  \left(
(\overline{Q}_{\psi,r}f)^{r}\right)  ^{\ast\ast}(t)\right]
^{1/r}\right\Vert _{X}\\
&  = \left\Vert
\frac{1}{\psi(t)}\,P_{r}(\overline{Q}_{\psi,r}f)(t)\right\Vert _{X}
\preceq\left\Vert
\frac{\overline{Q}_{\psi,r}f(t)}{\psi(t)}\right\Vert
_{X} \qquad\text{(by Lemma~\ref{Mpus}).}%
\end{align*}

Since
\begin{align*}
\left(  \frac{\overline{Q}_{\psi,r}f(t)}{\psi(t)}\right)  ^{r}  &  = \int%
_{t}^{1}\frac{\psi(s)^{r}}{\psi(t)^{r}}\,|f(s)|^{r}\,\frac{ds}{s}\\
&  = \int_{1}^{1/t}\frac{\psi(ut)^{r}}{\psi(t)^{r}}\,|f(ut)|^{r}\,\frac{du}%
{u}\\
&  \leq\int_{1}^{\infty}m_{\psi^{r}}(u)\,|f(ut)|^{r}\chi_{(0,1/u)}%
(t)\,\frac{du}{u},
\end{align*}
and $0<r\leq1$, the space $X^{(1/r)}$ is a Banach r.i.\ space, so
Minkowski's integral inequality yields
\begin{align*}
\left\Vert \left(  \frac{\overline{Q}_{\psi,r}f}{\psi}\right)
^{r}\right\Vert
_{X^{(1/r)}}  &  \leq\int_{1}^{\infty}m_{\psi^{r}}(u)\, \Vert\,|f|^{r}%
(u\cdot)\chi_{(0,1/u)}(\cdot)\Vert_{X^{(1/r)}}\,\frac{du}{u}\\
&  = \int_{1}^{\infty}m_{\psi^{r}}(u)\, \Vert E_{1/u}(|f|^{r})\Vert
_{X^{(1/r)}}\,\frac{du}{u}\\
&  \leq\left(
\int_{1}^{\infty}m_{\psi^{r}}(u)\,h_{X^{(1/r)}}(1/u)\,\frac
{du}{u}\right)  \Vert\,|f|^{r}\Vert_{X^{(1/r)}}.
\end{align*}

By (\ref{eq:boyd_lower_estimate_h}), applied to \(X^{(1/r)}\), for
every \(\varepsilon>0\) we have
\[
h_{X^{(1/r)}}(1/u) \preceq
u^{-\underline{\alpha}_{X^{(1/r)}}+\varepsilon} =
u^{-r\underline{\alpha}_{X}+\varepsilon}, \qquad u>1,
\]
where we used \((\ref{eq:boyd_dual_convexification})\). Moreover,
since $\psi\in\mathcal{A}_{0}$,
\[
m_{\psi^{r}}(u)\preceq u^{r\overline{\beta}_{\psi}+\varepsilon}.
\]
Choosing $0<\varepsilon<\frac r2(\underline{\alpha}_{X}-\overline{\beta}%
_{\psi})$, we get
\[
\int_{1}^{\infty}m_{\psi^{r}}(u)\,h_{X^{(1/r)}}(1/u)\,\frac{du}{u}<\infty.
\]
Hence
\[
\left\Vert \frac{\overline{Q}_{\psi,r}f}{\psi}\right\Vert _{X}^{r} =
\left\Vert \left(  \frac{\overline{Q}_{\psi,r}f}{\psi}\right)
^{r}\right\Vert _{X^{(1/r)}} \preceq\Vert\,|f|^{r}\Vert_{X^{(1/r)}}
= \Vert f\Vert_{X}^{r},
\]
which proves the boundedness of $\overline{Q}_{\psi,r}:X\rightarrow
Y$.

Therefore, by Theorem~\ref{inclu},
\[
\left\Vert \frac{\left(  (|f|^{r})^{\ast\ast}(\cdot)\right)  ^{1/r}}%
{\psi(\cdot)} \right\Vert _{X} \preceq\left\Vert \frac{O(|f|^{r},\cdot)^{1/r}%
}{\psi(\cdot)}\right\Vert _{X} +\|f\|_{L^{r}}.
\]
Combining the previous estimate with (\ref{localsub1}), we obtain
the first equivalence.

Finally, the second equivalence follows from \cite[Theorem
4.5]{Tur}, which for $0<r<1$ yields
\[
\left\Vert \frac{f^{\ast\ast}(t)}{\psi(t)}\right\Vert _{X}
\leq\left\Vert
\frac{((|f|^{r})^{\ast\ast}(t))^{1/r}}{\psi(t)}\right\Vert _{X}
\preceq \left\Vert \frac{f^{\ast\ast}(t)}{\psi(t)}\right\Vert _{X}.
\]
This completes the proof.
\end{proof}

\begin{remark}
An important feature of the subcritical regime is that, up to the
natural $L^{r}$ term, the oscillation space is independent of the
exponent $r$. More precisely,
\[
\left\Vert \frac{O(|f|^{r},\cdot)^{1/r}}{\psi(\cdot)}\right\Vert
_{X} +\|f\|_{L^{r}}
\]
is equivalent to
\[
\left\Vert \frac{f^{**}(\cdot)}{\psi(\cdot)}\right\Vert _{X}.
\]
Thus, in this regime, the oscillation functional is equivalent to a
maximal-type quantity.
\end{remark}

\


\subsection{The critical regime}

\label{sec:critical}

Throughout this subsection we assume that
\[
\left\Vert \frac{\psi(s)^{r}}{s}\chi_{(0,1)}(s)\right\Vert _{(X^{(1/r)}%
)^{\prime}} =\infty.
\]
By Theorem~\ref{supercrit}, this excludes the cases in which the
oscillation inequality already yields an embedding into
$L^{\infty}$.

In the critical situation the quotient $\psi/\varphi_{X}$ no longer
yields a purely power-type description. We therefore introduce the
function
\begin{equation}
M(t):=\sup_{t<s<1}\frac{\psi(s)}{\varphi_{X}(s)},\qquad0<t<1, \label{Mt}%
\end{equation}
which measures the maximal size of the quotient on intervals of the
form $(t,1)$ and will be referred to as the deviation function.

We now pass to the operator-theoretic formulation of the critical
case. As in the previous sections, the key point is the boundedness
of the Hardy-type operators $\overline{Q}_{\psi,r}$ and
$\overline{T}_{\psi,r}$. Recall that, for $1\leq\alpha<\infty$, a
couple $(X,Y)$ of r.i.\ spaces is called an $\alpha$-Berezhnoi pair
if $X$ satisfies an $\alpha$-lower estimate and $Y$ satisfies an
$\alpha$-upper estimate. The following criterion, which is the form
of Berezhnoi's theorem needed in this paper, is proved in
Appendix~\ref{app:berezhnoi_criterion}.

\begin{theorem}
\label{thm:Qpsi_rconc} Let $1\le\alpha<\infty$, and let $(X,Y)$ be
an $\alpha$-Berezhnoi pair of r.i.\ spaces. Then
$\overline{Q}_{\psi,r}:X\to Y$ is bounded if and only if
\begin{equation}
\sup_{0<x<1}\varphi_{Y^{(1/r)}}(x)\, \left\Vert \frac{\psi(s)^{r}}{s}%
\chi_{(x,1]}(s)\right\Vert _{(X^{(1/r)})^{\prime}} <\infty. \label{BerezhC}%
\end{equation}
\end{theorem}

We first derive the basic critical estimate in which the deviation
function $M$ and the logarithmic correction naturally appear.

\begin{theorem}
\label{thm:critical_general} Let $0<r\leq1$, let $X$ be an r.i.\
space satisfying an $\alpha$-lower estimate for some $\alpha>1$, and
let $\psi \in\mathcal{A}_{0}$. Set
\[
\beta^{\prime}:=\frac{\alpha}{\alpha-r}.
\]

\begin{enumerate}
\item There exists a constant $C>0$ such that for every measurable $f$ and
every $0<t<1$,
\begin{equation}
(|f|^{r})^{\ast\ast}(t)-(|f|^{r})^{\ast\ast}(1)\leq C\left(  \log\frac{e}%
{t}\right)  ^{1/\beta^{\prime}}\left\Vert \frac{O(|f|^{r},\cdot)^{1/r}}%
{\psi(\cdot)}\right\Vert _{X}^{r}M(t)^{r}. \label{eq:critical_general}%
\end{equation}

\item Let $Y$ be an r.i.\ space satisfying an $\alpha$-upper estimate. Assume
that
\[
\sup_{0<x<1}\varphi_{Y}(x)^{r}\left(  \log\frac{e}{x}\right)
^{1/\beta ^{\prime}}M(x)^{r}<\infty.
\]
Then there exists a constant $C>0$ such that for every measurable
$f$,
\[
\Vert f\Vert_{Y}\leq C\left(  \left\Vert \frac{O(|f|^{r},\cdot)^{1/r}}%
{\psi(\cdot)}\right\Vert _{X}+\Vert f\Vert_{L^{r}}\right)  .
\]

\end{enumerate}
\end{theorem}

\begin{proof}
\noindent\textbf{1)} By integrating (\ref{der2est}) over \((t,1)\),
we get
\[
\left(  |f|^{r}\right)  ^{\ast\ast}(t) =
\int_{t}^{1}O(|f|^{r},s)\,\frac {ds}{s} + \left(  |f|^{r}\right)
^{\ast\ast}(1), \qquad0<t<1.
\]
Hence
\[
\left(  |f|^{r}\right)  ^{\ast\ast}(t)-\left(  |f|^{r}\right)
^{\ast\ast}(1) = \int_{t}^{1}\frac{O(|f|^{r},s)}{\psi(s)^{r}}
\frac{\psi(s)^{r}}{s}\,ds.
\]

Since $X^{(1/r)}$ is an r.i.\ space, H\"older's inequality in the
pair $(X^{(1/r)},(X^{(1/r)})^{\prime})$ yields
\begin{align*}
\int_{t}^{1}\frac{O(|f|^{r},s)}{\psi(s)^{r}}
\frac{\psi(s)^{r}}{s}\,ds  & \le\left\Vert
\frac{O(|f|^{r},\cdot)}{\psi(\cdot)^{r}} \right\Vert _{X^{(1/r)}}
\left\Vert \frac{\psi(s)^{r}}{s}\chi_{(t,1]}(s) \right\Vert
_{(X^{(1/r)})^{\prime}}\\
&  = \left\Vert \frac{O(|f|^{r},\cdot)^{1/r}}{\psi(\cdot)}
\right\Vert _{X}^{r} \left\Vert \frac{\psi(s)^{r}}{s}\chi_{(t,1]}(s)
\right\Vert _{(X^{(1/r)})^{\prime}} .
\end{align*}

We now estimate the kernel
\[
\left\Vert \frac{\psi(s)^{r}}{s}\chi_{(t,1]}(s) \right\Vert _{(X^{(1/r)}%
)^{\prime}}, \qquad0<t<1.
\]

Since \(X\) satisfies an \(\alpha\)-lower estimate, it follows from
Remark~\ref{rem:convexification_estimates} that \(X^{(1/r)}\)
satisfies a \(\beta\)-lower estimate with
\[
\beta=\frac{\alpha}{r},
\]
and therefore \((X^{(1/r)})^{\prime}\) satisfies a
\(\beta^{\prime}\)-upper estimate, where
\[
\beta^{\prime}=\frac{\beta}{\beta-1} = \frac{\alpha}{\alpha-r}.
\]

Let $k\in\mathbb{N}$ be such that
\[
t\in(2^{-(k+1)},2^{-k}],
\]
and set
\[
I_{j}=(2^{-(j+1)},2^{-j}],\qquad j\ge0.
\]
Define
\[
w_{j}(s)=\frac{\psi(s)^{r}}{s}\chi_{I_{j}}(s), \qquad
W_{k}(s)=\sum_{j=0}^{k} w_{j}(s) =
\frac{\psi(s)^{r}}{s}\chi_{(2^{-(k+1)},1]}(s).
\]
Since $\chi_{(t,1]}\le\chi_{(2^{-(k+1)},1]}$, we obtain
\begin{equation}
\left\Vert \frac{\psi(s)^{r}}{s}\chi_{(t,1]}(s)\right\Vert _{(X^{(1/r)}%
)^{\prime}} \le\|W_{k}\|_{(X^{(1/r)})^{\prime}}. \label{eq:kernel_reduce}%
\end{equation}

The functions $w_{j}$ have pairwise disjoint supports and, since
$(X^{(1/r)})^{\prime}$ satisfies a $\beta^{\prime}$-upper estimate,
we have
\begin{equation}
\|W_{k}\|_{(X^{(1/r)})^{\prime}} \preceq\left(  \sum_{j=0}^{k} \|w_{j}%
\|_{(X^{(1/r)})^{\prime}}^{\beta^{\prime}} \right)
^{1/\beta^{\prime}}.
\label{eq:upper_estimate_dual}%
\end{equation}

For each $j$,
\[
\|w_{j}\|_{(X^{(1/r)})^{\prime}} \le\sup_{s\in
I_{j}}\frac{\psi(s)^{r}}{s}\,
\|\chi_{I_{j}}\|_{(X^{(1/r)})^{\prime}}.
\]
Since $2^{-(j+1)}\le s\le2^{-j}$ on $I_{j}$ and
$\psi\in\mathcal{A}_{0}$, we have
\[
\sup_{s\in I_{j}}\frac{\psi(s)^{r}}{s} \preceq\frac{\psi(2^{-j})^{r}%
}{2^{-(j+1)}}.
\]
Moreover, by (\ref{funddual}),
\[
\|\chi_{I_{j}}\|_{(X^{(1/r)})^{\prime}} = \frac{|I_{j}|}{\varphi_{X^{(1/r)}%
}(|I_{j}|)} = \frac{2^{-(j+1)}}{\varphi_{X}(2^{-(j+1)})^{r}}.
\]
Hence
\begin{equation}
\|w_{j}\|_{(X^{(1/r)})^{\prime}} \preceq\left(  \frac{\psi(2^{-j})}%
{\varphi_{X}(2^{-j})} \right)  ^{r}. \label{eq:block_estimate}%
\end{equation}

Substituting (\ref{eq:block_estimate}) into
(\ref{eq:upper_estimate_dual}), we get
\[
\|W_{k}\|_{(X^{(1/r)})^{\prime}} \preceq\left(  \sum_{j=0}^{k}
\left( \frac{\psi(2^{-j})}{\varphi_{X}(2^{-j})} \right)
^{r\beta^{\prime}} \right) ^{1/\beta^{\prime}}.
\]

Estimating the sum by the supremum gives
\begin{equation}
\|W_{k}\|_{(X^{(1/r)})^{\prime}} \preceq(k+1)^{1/\beta^{\prime}}
\sup_{0\le j\le k} \left(  \frac{\psi(2^{-j})}{\varphi_{X}(2^{-j})}
\right)  ^{r}.
\label{eq:kernel_prelog}%
\end{equation}

Since $t\in(2^{-(k+1)},2^{-k}]$, we have
\[
\{2^{-j}:0\le j\le k\}\subset(t,1),
\]
and therefore
\[
\sup_{0\le j\le k}\frac{\psi(2^{-j})}{\varphi_{X}(2^{-j})} \le\sup
_{t<u<1}\frac{\psi(u)}{\varphi_{X}(u)} = M(t).
\]
Hence (\ref{eq:kernel_prelog}) yields
\[
\|W_{k}\|_{(X^{(1/r)})^{\prime}}
\preceq(k+1)^{1/\beta^{\prime}}\,M(t)^{r}.
\]

Finally, since $t\in(2^{-(k+1)},2^{-k}]$, one has
\[
k\log2\le\log\frac1t < (k+1)\log2,
\]
and therefore
\[
k+1\simeq\log\frac{e}{t}.
\]
Thus
\[
\|W_{k}\|_{(X^{(1/r)})^{\prime}} \preceq(\log(e/t))^{1/\beta^{\prime}%
}\,M(t)^{r}.
\]
Combining this with (\ref{eq:kernel_reduce}), we obtain
\[
\left\Vert \frac{\psi(s)^{r}}{s}\chi_{(t,1]}(s)\right\Vert _{(X^{(1/r)}%
)^{\prime}} \preceq(\log(e/t))^{1/\beta^{\prime}}\,M(t)^{r}.
\]

Substituting this estimate into the previous H\"older inequality
yields
\[
(|f|^{r})^{**}(t)-(|f|^{r})^{**}(1) \preceq(\log(e/t))^{1/\beta^{\prime}%
}M(t)^{r} \left\Vert
\frac{O(|f|^{r},\cdot)^{1/r}}{\psi(\cdot)}\right\Vert _{X}^{r},
\]
which proves (\ref{eq:critical_general}).

\medskip

\noindent\textbf{2)} By assumption, $X$ satisfies an $\alpha$-lower
estimate and $Y$ satisfies an $\alpha$-upper estimate. Hence $(X,Y)$
is an $\alpha$-Berezhnoi pair, and therefore the convexified couple
$(X^{(1/r)},Y^{(1/r)})$ is an $(\alpha/r)$-Berezhnoi pair. Consider
the Hardy-type operator
\[
\overline T_{\psi,r}g(t)=\int_{t}^{1} \psi(s)^{r}
g(s)\,\frac{ds}{s}, \qquad0<t<1.
\]
By Theorem~\ref{inclu},
\[
\overline Q_{\psi,r}:X\to Y \quad\text{bounded} \Longleftrightarrow
\quad\overline T_{\psi,r}:X^{(1/r)}\to Y^{(1/r)}
\quad\text{bounded}.
\]

Hence, by Theorem~\ref{thm:Qpsi_rconc}, the boundedness of
$\overline Q_{\psi,r}$ follows once we verify
\begin{equation}
\sup_{0<x<1} \varphi_{Y^{(1/r)}}(x) \left\|
\frac{\psi(s)^{r}}{s}\chi _{(x,1]}(s) \right\|
_{(X^{(1/r)})^{\prime}} <\infty.
\label{eq:Berezhnoi_check}%
\end{equation}

From the proof of part~\textup{(1)} we already know that
\[
\left\|  \frac{\psi(s)^{r}}{s}\chi_{(x,1]}(s) \right\|
_{(X^{(1/r)})^{\prime}
} \preceq(\log(e/x))^{1/\beta^{\prime}}\,M(x)^{r}, \qquad\beta^{\prime}%
=\frac{\alpha}{\alpha-r}.
\]
Substituting this estimate into (\ref{eq:Berezhnoi_check}), we
obtain
\[
\sup_{0<x<1}
\varphi_{Y^{(1/r)}}(x)\,(\log(e/x))^{1/\beta^{\prime}}\,M(x)^{r}
<\infty.
\]
Since
\[
\varphi_{Y^{(1/r)}}(x)=\varphi_{Y}(x)^{r},
\]
this condition is precisely
\[
\sup_{0<x<1}
\varphi_{Y}(x)^{r}\,(\log(e/x))^{1/\beta^{\prime}}\,M(x)^{r}
<\infty,
\]
which holds by assumption. Therefore
\[
\overline T_{\psi,r}:X^{(1/r)}\to Y^{(1/r)}
\]
is bounded, and hence so is
\[
\overline Q_{\psi,r}:X\to Y.
\]

Finally, by the implication $(iii)\Rightarrow(i)$ in
Theorem~\ref{inclu},
\[
\Vert f\Vert_{Y}\preceq\left\Vert
\frac{O(|f|^{r},\cdot)^{1/r}}{\psi(\cdot )}\right\Vert _{X}+\Vert
f\Vert_{L^{r}},
\]
which completes the proof.
\end{proof}

\medskip\noindent\textbf{Trudinger-type embeddings.}

The estimates in Theorem~\ref{thm:critical_general} depend on the
asymptotic behaviour of the deviation function \(M\). We first
consider the case in which \(M\) is bounded. This corresponds to the
upper critical regime and leads to Trudinger-type integrability.

\begin{theorem}
Let \(0<r\leq1\). Let \(X\) be an r.i.\ space satisfying an
\(\alpha\)-lower estimate for some \(\alpha>1\), let
\(\psi\in\mathcal{A}_{0}\), and set
\[
\beta^{\prime}=\frac{\alpha}{\alpha-r}.
\]
Assume that the deviation function \(M\), defined in \((\ref{Mt})\),
is bounded. Then there exist constants \(c,C>0\) such that, for
every measurable \(f\) such that \(A<\infty\) and \(f\in L^r\),
\begin{equation}
\int_{0}^{1} \exp\!\left( c\left(
\frac{(|f|^{r})^{**}(t)-(|f|^{r})^{**}(1)}{A}
\right)^{\beta^{\prime}} \right)\,dt \le C,
\label{eq:trudinger_exp_fr_fss}
\end{equation}
where
\[
A= \left\Vert \frac{O(|f|^{r},\cdot)^{1/r}}{\psi(\cdot)} \right\Vert
_{X}^{r}, \qquad (|f|^{r})^{**}(1)=\int_{0}^{1}(|f|^{r})^{*}(s)\,ds.
\]
Here, if \(A=0\), the quotient in
\((\ref{eq:trudinger_exp_fr_fss})\) is understood in the trivial
limiting sense.

Moreover,
\[
(|f|^{r})^{**}(t)-(|f|^{r})^{**}(1) \in \exp L^{\beta^{\prime}},
\]
with
\[
\left\Vert (|f|^{r})^{**}(\cdot)-(|f|^{r})^{**}(1) \right\Vert_{\exp
L^{\beta^{\prime}}} \preceq A.
\]
Finally, if
\[
m:=(|f|^{r})^{**}(1),
\]
then
\begin{equation}
\int_{0}^{1} \exp\!\left( c\left( \frac{(|f|^{r})^{**}(t)}{A+m}
\right)^{\beta^{\prime}} \right)\,dt \le C .
\label{eq:trudinger_exp_shifted_fss}
\end{equation}
\end{theorem}

\begin{proof}
If \(A=0\), then the estimate below gives
\[
(|f|^r)^{**}(t)=(|f|^r)^{**}(1), \qquad 0<t<1,
\]
and the conclusion is immediate. Hence assume \(A>0\).

Since \(M\) is bounded, Theorem~\ref{thm:critical_general} gives,
for \(0<t<1\),
\begin{equation}
(|f|^{r})^{**}(t)-(|f|^{r})^{**}(1) \le C_{0}\, \left(\log\frac
et\right)^{1/\beta^{\prime}} \left\Vert
\frac{O(|f|^{r},\cdot)^{1/r}}{\psi(\cdot)} \right\Vert_{X}^{r}.
\label{eq:crit_bound_fss_C0}
\end{equation}
Set
\[
m=(|f|^{r})^{**}(1), \qquad H(t)=(|f|^{r})^{**}(t)-m.
\]
Then \((\ref{eq:crit_bound_fss_C0})\) yields
\[
H(t)\le C_{0}A \left(\log\frac et\right)^{1/\beta^{\prime}}, \qquad
0<t<1.
\]
Hence
\[
\left(\frac{H(t)}{A}\right)^{\beta^{\prime}} \le
C_{0}^{\beta^{\prime}}\log\frac et, \qquad 0<t<1.
\]
Choose \(c>0\) such that
\[
cC_{0}^{\beta^{\prime}}<1.
\]
Then
\[
\exp\!\left( c\left(\frac{H(t)}{A}\right)^{\beta^{\prime}} \right)
\le \left(\frac et\right)^{cC_{0}^{\beta^{\prime}}}.
\]
Since \(cC_{0}^{\beta^{\prime}}<1\), the right-hand side is
integrable on \((0,1)\). This proves
\((\ref{eq:trudinger_exp_fr_fss})\).

The corresponding \(\exp L^{\beta'}\)-norm estimate is just the
standard Luxemburg-norm reformulation of
\((\ref{eq:trudinger_exp_fr_fss})\) (see \cite{BS}).

It remains to prove the shifted estimate. Since
\[
(|f|^{r})^{**}(t)=H(t)+m,
\]
and since \(\beta'>0\), there is a constant \(C_{\beta'}\) such that
\[
(H(t)+m)^{\beta'} \le
C_{\beta'}\bigl(H(t)^{\beta'}+m^{\beta'}\bigr).
\]
Therefore
\[
\left( \frac{(|f|^{r})^{**}(t)}{A+m} \right)^{\beta'} \le C_{\beta'}
\left[ \left(\frac{H(t)}{A+m}\right)^{\beta'} +
\left(\frac{m}{A+m}\right)^{\beta'} \right].
\]
Since \(A\le A+m\) and \(m/(A+m)\le1\), we get
\[
\left( \frac{(|f|^{r})^{**}(t)}{A+m} \right)^{\beta'} \le C_{\beta'}
\left[ \left(\frac{H(t)}{A}\right)^{\beta'} +1 \right].
\]
Taking \(c>0\) smaller if necessary, depending only on \(\beta'\)
and the previous constants, the estimate
\((\ref{eq:trudinger_exp_shifted_fss})\) follows from
\((\ref{eq:trudinger_exp_fr_fss})\).
\end{proof}

\medskip\noindent\textbf{Hansson-type embeddings.}

We now describe the logarithmic endpoint targets naturally
associated with the critical regime. Following the terminology
motivated by Hansson's classical endpoint embedding \cite{Hansson},
we call these spaces Hansson-type targets.

Throughout the rest of this subsection, we restrict ourselves to the
natural critical case in which the deviation function \(M\), defined
in \((\ref{Mt})\), carries no residual power contribution, namely
\[
\underline{\beta}_M=\overline{\beta}_M=0. \label{Hzeroind}
\]
This assumption is what leads to a purely logarithmic endpoint
correction.

\begin{theorem}
\label{thm:hansson_Lalpha} Let \(0<r\le1\), let \(\psi\in\mathcal
A_0\), and let \(X\) be an r.i.\ space satisfying an
\(\alpha\)-lower estimate for some \(\alpha>1\). Let \(\mathcal
H_{\alpha,r,M}\) be the space defined by
\[
\mathcal H_{\alpha,r,M} := \left\{ f\in L^0: \|f\|_{\mathcal
H_{\alpha,r,M}}<\infty \right\},
\]
where
\[
\|f\|_{\mathcal H_{\alpha,r,M}} := \left( \int_0^1 \left(
\frac{f^{**}(t)} {(\log(e/t))^{1/r}M(t)} \right)^\alpha \frac{dt}{t}
\right)^{1/\alpha}.
\]
Then there exists a constant \(C>0\) such that, for every measurable
\(f\),
\begin{equation}
\|f\|_{\mathcal H_{\alpha,r,M}} \le C\left( \left\|
\frac{O(|f|^r,\cdot)^{1/r}}{\psi(\cdot)} \right\|_X + \|f\|_{L^r}
\right). \label{eq:hansson_Lalpha_embedding}
\end{equation}
\end{theorem}

\begin{proof}
Set
\[
H:=\mathcal H_{\alpha,r,M}, \qquad L(t):=(\log(e/t))^{1/r}.
\]

We first prove that the norm defining \(H\) is equivalent to its
\(f^*\)-version. Since \(f^*\le f^{**}\), one immediately has
\[
\left( \int_0^1 \left( \frac{f^*(t)}{L(t)M(t)} \right)^\alpha
\frac{dt}{t} \right)^{1/\alpha} \le \|f\|_H.
\]

We prove the converse inequality. Since \(L(t)=(\log(e/t))^{1/r}\)
has vanishing fundamental indices and, by \((\ref{Hzeroind})\),
\[
\underline\beta_M=\overline\beta_M=0,
\]
the product \(LM\) also has vanishing fundamental indices.  Hence,
by the definition of the fundamental indices \((\ref{fundind})\),
the function
\[
A(t):=t^{1/\alpha}L(t)M(t)
\]
has lower fundamental index \(1/\alpha>0\). In particular, \(A\) is
quasi-increasing on \((0,1/2)\). Thus, for \(0<s<t<1/2\),
\[
A(s)\preceq A(t).
\]
Equivalently,
\[
\frac{1}{L(t)M(t)} \preceq \left(\frac{t}{s}\right)^{1/\alpha}
\frac{1}{L(s)M(s)}, \qquad 0<s<t<1/2.
\]
Therefore, for \(0<t<1/2\),
\begin{align*}
\frac{f^{**}(t)}{t^{1/\alpha}L(t)M(t)} &=
\frac{1}{t^{1+1/\alpha}L(t)M(t)}
\int_0^t f^*(s)\,ds  \\
&\preceq \frac1t \int_0^t \frac{f^*(s)} {s^{1/\alpha}L(s)M(s)} \,ds.
\end{align*}
Consequently,
\begin{align*}
\int_0^{1/2} \left( \frac{f^{**}(t)} {L(t)M(t)} \right)^\alpha
\frac{dt}{t} &= \int_0^{1/2} \left( \frac{f^{**}(t)}
{t^{1/\alpha}L(t)M(t)} \right)^\alpha
dt  \\
&\preceq \int_0^{1/2} \left( \frac1t \int_0^t \frac{f^*(s)}
{s^{1/\alpha}L(s)M(s)} \,ds \right)^\alpha dt .
\end{align*}
By Hardy's inequality in \(L^\alpha(0,1)\) (see \cite{BS}),
\[
\int_0^{1/2} \left( \frac1t \int_0^t \frac{f^*(s)}
{s^{1/\alpha}L(s)M(s)} \,ds \right)^\alpha dt \preceq \int_0^{1/2}
\left( \frac{f^*(t)} {t^{1/\alpha}L(t)M(t)} \right)^\alpha dt.
\]
Thus
\[
\int_0^{1/2} \left( \frac{f^{**}(t)} {L(t)M(t)} \right)^\alpha
\frac{dt}{t} \preceq \int_0^{1/2} \left( \frac{f^*(t)} {L(t)M(t)}
\right)^\alpha \frac{dt}{t}.
\]

It remains to estimate the part on \((1/2,1)\). On each interval
away from the origin, the functions \(L\) and \(M\) are equivalent
to positive constants. Since \(f^{**}\) is decreasing,
\[
\left( \int_{1/2}^1 \left( \frac{f^{**}(t)} {L(t)M(t)}
\right)^\alpha \frac{dt}{t} \right)^{1/\alpha} \preceq f^{**}(1/2).
\]
The same is true on \((1/4,1/2)\), and \(f^{**}(t)\ge f^{**}(1/2)\)
for \(t\in(1/4,1/2)\). Hence
\[
f^{**}(1/2) \preceq \left( \int_{1/4}^{1/2} \left( \frac{f^{**}(t)}
{L(t)M(t)} \right)^\alpha \frac{dt}{t} \right)^{1/\alpha}.
\]
Using the estimate already proved on \((0,1/2)\), we get
\[
f^{**}(1/2) \preceq \left( \int_0^1 \left( \frac{f^*(t)} {L(t)M(t)}
\right)^\alpha \frac{dt}{t} \right)^{1/\alpha}.
\]
Therefore
\[
\left( \int_{1/2}^1 \left( \frac{f^{**}(t)} {L(t)M(t)}
\right)^\alpha \frac{dt}{t} \right)^{1/\alpha} \preceq \left(
\int_0^1 \left( \frac{f^*(t)} {L(t)M(t)} \right)^\alpha \frac{dt}{t}
\right)^{1/\alpha}.
\]

Combining the estimates on \((0,1/2)\) and \((1/2,1)\), we obtain
\begin{equation}
\|f\|_H \simeq \left( \int_0^1 \left( \frac{f^*(t)}
{(\log(e/t))^{1/r}M(t)} \right)^\alpha \frac{dt}{t}
\right)^{1/\alpha}. \label{eq:Halpha_star_equiv}
\end{equation}

In particular, the norm of \(H\) is equivalent to an
\(L^\alpha(dt/t)\)-norm with weight, and therefore \(H\) satisfies
an \(\alpha\)-upper estimate.

We now verify the fundamental-function condition in
Theorem~\ref{thm:critical_general}\textup{(2)}. Let
\[
\beta'=\frac{\alpha}{\alpha-r}.
\]
By \((\ref{eq:Halpha_star_equiv})\), for \(0<x<1\),
\[
\varphi_H(x) \simeq \left( \int_0^x \frac{dt}
{t(\log(e/t))^{\alpha/r}M(t)^\alpha} \right)^{1/\alpha}.
\]
Since \(M\) is decreasing,
\[
\frac1{M(t)^\alpha} \le \frac1{M(x)^\alpha}, \qquad 0<t<x.
\]
Therefore
\[
\varphi_H(x) \preceq \frac1{M(x)} \left( \int_0^x \frac{dt}
{t(\log(e/t))^{\alpha/r}} \right)^{1/\alpha}.
\]
As \(\alpha/r>1\),
\[
\int_0^x \frac{dt} {t(\log(e/t))^{\alpha/r}} \simeq
(\log(e/x))^{1-\alpha/r}.
\]
Consequently,
\begin{equation*}
\varphi_H(x) \preceq \frac1{(\log(e/x))^{1/r-1/\alpha}M(x)}.
\end{equation*}
Raising this estimate to the power \(r\), we get
\[
\varphi_H(x)^r \preceq \frac1{(\log(e/x))^{1-r/\alpha}M(x)^r}.
\]
Since
\[
\frac1{\beta'}=1-\frac r\alpha,
\]
it follows that
\begin{equation}
\sup_{0<x<1} \varphi_H(x)^r (\log(e/x))^{1/\beta'} M(x)^r <\infty.
\label{eq:Halpha_condition}
\end{equation}

We have proved that \(X\) satisfies an \(\alpha\)-lower estimate by
hypothesis, that \(H\) satisfies an \(\alpha\)-upper estimate, and
that \((\ref{eq:Halpha_condition})\) holds. Hence
Theorem~\ref{thm:critical_general}\textup{(2)} applies with this
choice of \(H\), and gives
\[
\|f\|_H \preceq \left\| \frac{O(|f|^r,\cdot)^{1/r}}{\psi(\cdot)}
\right\|_X + \|f\|_{L^r}.
\]
Since \(H=\mathcal H_{\alpha,r,M}\), this is precisely
\((\ref{eq:hansson_Lalpha_embedding})\).
\end{proof}

The preceding theorem gives a canonical Hansson-type endpoint which
is always available. We now introduce a second family of targets,
still adapted to the geometry of \(X\), by allowing a
convexification of the underlying space. The case \(s=1\), which
corresponds to \(\rho=\alpha\), gives the intrinsic Hansson-type
target associated with \(X\).

We shall also use the following Sharpley-type estimate, related to
the spaces \(\Lambda^p(X)\) introduced in \cite{Shar}. Its proof is
included in Appendix~\ref{app:sharpley}.

\begin{lemma}
\label{Sharpley} Let \(X\) be an r.i.\ space satisfying a
\(p\)-upper estimate for some \(p>1\). Define
\[
\|f\|_{\Lambda^{p}(X)} := \left( \int_0^1
\left(f^*(t)\varphi_X(t)\right)^p \frac{dt}{t} \right)^{1/p}.
\]
Then
\[
\|f\|_X \preceq \|f\|_{\Lambda^p(X)}, \qquad f\in L^0.
\]
\end{lemma}

\begin{definition}
Let \(X\) be an r.i.\ space, let \(s\ge1\), and let \(M\) be the
deviation function defined in \((\ref{Mt})\). The \(s\)-convexified
Hansson-type space associated with \(X\) is defined by
\[
H_{X,s,r,M} := \left\{ f\in L^0: \|f\|_{H_{X,s,r,M}}<\infty
\right\},
\]
where
\[
\|f\|_{H_{X,s,r,M}} := \left\| \frac{f^{**}(\cdot)}
{\varphi_{X^{(s)}}(\cdot)(\log(e/\cdot))^{1/r}M(\cdot)}
\right\|_{X^{(s)}} .
\]
When \(s=1\), we simply write
\[
H_{X,r,M}:=H_{X,1,r,M};
\]
this is the intrinsic Hansson-type target associated with \(X\).
\end{definition}

\begin{theorem}
\label{thm:hansson_convexified} Let \(0<r\le1\), let
\(\psi\in\mathcal A_0\), and let \(X\) be an r.i.\ space satisfying
an \(\alpha\)-lower estimate and a \(\rho\)-upper estimate, with
\(1<\rho\le\alpha\). Set
\[
s:=\frac{\alpha}{\rho}.
\]
Then there exists a constant \(C>0\) such that, for every measurable
\(f\),
\[
\|f\|_{H_{X,s,r,M}} \le C\left( \left\|
\frac{O(|f|^r,\cdot)^{1/r}}{\psi(\cdot)} \right\|_X + \|f\|_{L^r}
\right).
\]
\end{theorem}

\begin{proof}
Set
\[
Y:=X^{(s)}.
\]
Since \(X\) satisfies a \(\rho\)-upper estimate and
\(s=\alpha/\rho\), it follows from
Remark~\ref{rem:convexification_estimates} that \(Y\) satisfies an
\(\alpha\)-upper estimate.

Let
\[
L(t):=(\log(e/t))^{1/r}, \qquad w(t):=\varphi_Y(t)L(t)M(t).
\]
Since \(X\) satisfies an \(\alpha\)-lower estimate, testing this
estimate on characteristic functions of pairwise disjoint sets shows
that \(\varphi_X\) has positive lower fundamental index. Indeed,
applying the \(\alpha\)-lower estimate to \(n\) disjoint sets of
measure \(t/n\) gives
\[
n^{1/\alpha}\varphi_X(t/n)\le C\varphi_X(t).
\]
Equivalently,
\[
\varphi_X(\lambda t)\preceq \lambda^{1/\alpha}\varphi_X(t), \qquad
0<\lambda<1,
\]
and hence
\[
m_{\varphi_X}(\lambda)\preceq \lambda^{1/\alpha},\qquad
\lambda\to0^+.
\]
Thus
\[
\underline{\beta}_{\varphi_X}\ge \frac1\alpha.
\]

Therefore
\[
\varphi_Y(t)=\varphi_{X^{(s)}}(t)=\varphi_X(t)^{1/s}
\]
also has positive lower fundamental index. Since \(L\) and \(M\)
have vanishing fundamental indices, the function \(w\) has positive
lower fundamental index. In particular, \(w\) is quasi-increasing on
\((0,1/2)\).

Since \(Y\) satisfies an \(\alpha\)-upper estimate with
\(\alpha>1\), one has
\[
\overline{\alpha}_{Y}\le \frac1\alpha<1
\]
(see, e.g., \cite{LT2,BS}). Hence the Hardy operator
\[
Pg(t):=\frac1t\int_0^t g(s)\,ds
\]
is bounded on \(Y\). Therefore the usual Hardy argument gives
\[
\left\| \frac{f^{**}}{w} \right\|_Y \simeq \left\| \frac{f^*}{w}
\right\|_Y .
\]
Consequently,
\[
\|f\|_{H_{X,s,r,M}} \simeq \left\| \frac{f^*}{w} \right\|_Y .
\]

Since \(Y\) satisfies an \(\alpha\)-upper estimate,
Lemma~\ref{Sharpley} gives
\[
\left\| \frac{f^*}{w} \right\|_Y \preceq \left( \int_0^1 \left[
\left(\frac{f^*}{w}\right)^*(t)\varphi_Y(t) \right]^\alpha
\frac{dt}{t} \right)^{1/\alpha}.
\]
Since \(f^*\) is decreasing and \(w\) is quasi-increasing, the
function \(f^*/w\) is equivalent to a decreasing function. Hence
\[
\left(\frac{f^*}{w}\right)^*(t) \preceq \frac{f^*(t)}{w(t)}.
\]
Therefore
\[
\|f\|_{H_{X,s,r,M}} \preceq \left( \int_0^1 \left(
\frac{f^*(t)\varphi_Y(t)} {\varphi_Y(t)(\log(e/t))^{1/r}M(t)}
\right)^\alpha \frac{dt}{t} \right)^{1/\alpha}.
\]
Cancelling \(\varphi_Y(t)\), we obtain
\[
\|f\|_{H_{X,s,r,M}} \preceq \left( \int_0^1 \left( \frac{f^*(t)}
{(\log(e/t))^{1/r}M(t)} \right)^\alpha \frac{dt}{t}
\right)^{1/\alpha}.
\]
By Theorem~\ref{thm:hansson_Lalpha}, the last expression is bounded
by
\[
C\left( \left\| \frac{O(|f|^r,\cdot)^{1/r}}{\psi(\cdot)} \right\|_X
+ \|f\|_{L^r} \right).
\]
This proves the desired estimate.
\end{proof}


\section{Extension to the quasi-Banach setting}
\label{sec:quasi}

We briefly indicate how the previous results extend to quasi-Banach
rearrangement-invariant spaces.

\begin{definition}
Let \(0<r\le1\). A quasi-Banach r.i.\ space \(X\) is said to be
\emph{\(r\)-convex} if the convexified space \(X^{(1/r)}\) is a
Banach r.i.\ space.
\end{definition}

Although there exist quasi-Banach r.i.\ spaces that fail to be
\(r\)-convex for every \(0<r\le1\) (see \cite{JS}), such examples
are exceptional. In fact, as observed by Grafakos and Kalton, ``all
practical quasi-Banach rearrangement-invariant spaces are
\(r\)-convex for some \(0<r\le1\)'' (see \cite{GK}). For this
reason, we restrict ourselves to \(r\)-convex quasi-Banach spaces.

Let \(X\) be an \(r\)-convex quasi-Banach r.i.\ space and set
\[
\widetilde X:=X^{(1/r)}.
\]
Then \(\widetilde X\) is a Banach r.i.\ space. The oscillation
condition for \(f\) in \(X\) can be rewritten as a Banach condition
for
\[
g:=|f|^r.
\]
Indeed,
\[
O(g,t)=O(|f|^r,t),
\]
and
\[
\left\| \frac{O(|f|^r,\cdot)^{1/r}}{\psi(\cdot)} \right\|_X^r =
\left\| \frac{O(g,\cdot)}{\psi(\cdot)^r} \right\|_{\widetilde X}.
\]
Thus the pair \((X,\psi)\) in the quasi-Banach setting is
transformed into the Banach pair
\[
(\widetilde X,\widetilde\psi), \qquad \widetilde\psi:=\psi^r,
\]
with oscillation exponent equal to \(1\). Notice that
\(\widetilde\psi\in\mathcal A_0\), since
\[
\underline{\beta}_{\widetilde\psi} = r\,\underline{\beta}_{\psi},
\qquad \overline{\beta}_{\widetilde\psi} =
r\,\overline{\beta}_{\psi}.
\]

The relevant structural assumptions are transformed accordingly. If
\(X\) satisfies an \(\alpha\)-lower estimate, then \(\widetilde
X=X^{(1/r)}\) satisfies an \(\alpha/r\)-lower estimate. Similarly,
if \(X\) satisfies a \(\rho\)-upper estimate, then \(\widetilde X\)
satisfies a \(\rho/r\)-upper estimate.

Therefore each of the Banach results proved above applies to
\((\widetilde X,\widetilde\psi)\), with oscillation exponent \(1\),
provided the corresponding lower and upper estimate assumptions are
satisfied by \(\widetilde X\). If \(\widetilde Y\) is the Banach
target obtained in this way, the corresponding quasi-Banach target
is defined by
\[
Y:=\widetilde Y^{(r)}.
\]
Then
\[
\|f\|_Y^r=\||f|^r\|_{\widetilde Y}.
\]
Consequently, an estimate of the form
\[
\|g\|_{\widetilde Y} \le C\left( \left\|
\frac{O(g,\cdot)}{\widetilde\psi(\cdot)} \right\|_{\widetilde X} +
\|g\|_{L^1} \right)
\]
yields, with \(g=|f|^r\),
\[
\|f\|_Y \le C\left( \left\| \frac{O(|f|^r,\cdot)^{1/r}}{\psi(\cdot)}
\right\|_X + \|f\|_{L^r} \right),
\]
with a possibly different constant.

In this way the supercritical, subcritical and critical embeddings
extend to the \(r\)-convex quasi-Banach setting. In the critical
case, the Hansson and Trudinger targets are obtained by applying the
Banach results to \((\widetilde X,\psi^r)\), with oscillation
exponent \(1\), and then deconvexifying the resulting target.

\section{Auxiliary proofs}

\subsection{Proof of Lemma~\ref{Mpus}}

\begin{proof}
Since $P_{r}$ depends only on $|f|$, we may assume $f\geq0$. For
$t>0$,
\[
\left(  \frac{P_{r}f(t)}{\psi(t)}\right)  ^{r}=\frac{1}{t\,\psi(t)^{r}}%
\int_{0}^{t}f(s)^{r}\,ds=\int_{0}^{1}\left(
\frac{f(st)}{\psi(t)}\right) ^{r}\,ds.
\]
Since $0<r\leq1$, Jensen's inequality for the concave function
$x\mapsto x^{r}$ yields
\[
\left(  \frac{P_{r}f(t)}{\psi(t)}\right)  ^{r}\leq\left(  \int_{0}^{1}%
\frac{f(st)}{\psi(t)}\,ds\right)  ^{r}.
\]
Hence
\[
\frac{P_{r}f(t)}{\psi(t)}\leq\int_{0}^{1}\frac{f(st)}{\psi(t)}\,ds
\leq\int_{0}^{1}\frac{\psi(st)}{\psi(t)}\,\frac{f(st)}{\psi(st)}\,ds
\leq\int_{0}^{1}m_{\psi}(s)\,\frac{f(st)}{\psi(st)}\,ds.
\]

Taking the $Y$-norm and using Minkowski's integral inequality,
\[
\left\Vert \frac{P_{r}f}{\psi}\right\Vert
_{Y}\leq\int_{0}^{1}m_{\psi }(s)\left\Vert
\frac{f(s\cdot)}{\psi(s\cdot)}\right\Vert _{Y}\,ds,
\]
which implies
\[
\left\Vert \frac{f(s\cdot)}{\psi(s\cdot)}\right\Vert _{Y}\leq h_{Y}%
(1/s)\left\Vert \frac{f}{\psi}\right\Vert _{Y},
\]
where $h_{Y}$ denotes the dilation function of $Y$. Since
$h_{Y}(1/s)\leq \max\{1,1/s\}$ (see \cite{BS,KPS}), we obtain
\[
\left\Vert \frac{P_{r}f}{\psi}\right\Vert _{Y}\leq\left(
\int_{0}^{1}m_{\psi
}(s)\,h_{Y}(1/s)\,ds\right)  \left\Vert \frac{f}{\psi}\right\Vert _{Y}%
\leq\left(  \int_{0}^{1}\frac{m_{\psi}(s)}{s}\,ds\right)  \left\Vert
\frac {f}{\psi}\right\Vert _{Y}.
\]

Finally, since $\psi\in\mathcal{A}_{0}$ we have
$\underline{\beta}_{\psi}>0$, hence $m_{\psi}(s)\leq
C_{\varepsilon}s^{\underline{\beta}_{\psi}-\varepsilon }$ for any
$\underline{\beta}_{\psi}>\varepsilon>0$. Therefore
\[
\int_{0}^{1}\frac{m_{\psi}(s)}{s}\,ds<\infty,
\]
and the proof is complete.
\end{proof}

\subsection{Berezhnoi's criterion for localized Hardy operators}
\label{app:berezhnoi_criterion}

\begin{definition}
Let $1\le\alpha<\infty$. Let $X$ and $Y$ be r.i.\ spaces. The couple
$(X,Y)$ is called an \emph{$\alpha$-Berezhnoi pair} if $X$ satisfies
an $\alpha$-lower estimate and $Y$ satisfies an $\alpha$-upper
estimate.
\end{definition}

The following result is a reformulation of Berezhnoi's
characterization of the boundedness of Hardy-type operators between
rearrangement-invariant spaces; see
\cite{Berezhnoi1,Berezhnoi2,Berezhnoi3}.

\begin{theorem}
Let $X$ and $Y$ be r.i.\ spaces, and let
\[
Tf(t)=\int_{t}^{1} K(t,s)\,f(s)\,\frac{ds}{s},\qquad0<t<1,
\]
where $K(t,s)\ge0$ is measurable for $0<t<s<1$. Assume that $(X,Y)$
is an $\alpha$-Berezhnoi pair for some $1\le\alpha<\infty$. Then the
following statements are equivalent:

\begin{enumerate}
\item $T:X\to Y$ is bounded;

\item
\[
\sup_{0<x<1}\varphi_{Y}(x)\, \left\Vert
K(x,s)\chi_{(x,1]}(s)\right\Vert _{X^{\prime}}<\infty.
\]

\end{enumerate}
\end{theorem}

We now specialize this criterion to the operator
$\overline{Q}_{\psi,r}$.

\begin{theorem}
Let $1\leq\alpha<\infty$, and let $(X,Y)$ be an $\alpha$-Berezhnoi
pair of r.i.\ spaces. Then $\overline{Q}_{\psi,r}:X\rightarrow Y$ is
bounded if and only if
\[
\sup_{0<x<1}\varphi_{Y^{(1/r)}}(x)\,\left\Vert \frac{\psi(s)^{r}}{s}%
\chi_{(x,1]}(s)\right\Vert _{(X^{(1/r)})^{\prime}}<\infty.
\]

\end{theorem}

\begin{proof}
Since $\overline{Q}_{\psi,r}$ is not linear, we pass to the
associated Hardy-type operator
\[
\overline{T}_{\psi,r}g(t)=\int_{t}^{1}\psi(s)^{r}
g(s)\,\frac{ds}{s}, \qquad0<t<1.
\]
As in Theorem~\ref{inclu},
\begin{equation}
\overline{Q}_{\psi,r}:X\to Y \text{ is bounded}
\quad\Longleftrightarrow \quad\overline{T}_{\psi,r}:X^{(1/r)}\to
Y^{(1/r)} \text{ is bounded}.
\label{Bere1}%
\end{equation}

Since $(X,Y)$ is an $\alpha$-Berezhnoi pair, the convexified couple
$(X^{(1/r)},Y^{(1/r)})$ is an $(\alpha/r)$-Berezhnoi pair. Hence
Berezhnoi's criterion applies to $\overline{T}_{\psi,r}$ and yields
that
\[
\overline{T}_{\psi,r}:X^{(1/r)}\to Y^{(1/r)}
\]
is bounded if and only if
\[
\sup_{0<x<1}\varphi_{Y^{(1/r)}}(x)\, \left\Vert \frac{\psi(s)^{r}}{s}%
\chi_{(x,1]}(s)\right\Vert _{(X^{(1/r)})^{\prime}} <\infty.
\]
Combining this with (\ref{Bere1}) gives (\ref{BerezhC}).
\end{proof}

\subsection{Proof of Lemma~\ref{Sharpley}}
\label{app:sharpley}

\begin{proof}
Since \(X\) is rearrangement-invariant, we may assume that
\(f=f^*\). Let
\[
I_k=(2^{-(k+1)},2^{-k}],\qquad k=0,1,2,\ldots .
\]
Since \(f^*\) is decreasing, for \(t\in I_k\) we have
\[
f^*(t)\le f^*(2^{-(k+1)}).
\]
Hence
\[
f^*(t) \le \sum_{k=0}^{\infty} f^*(2^{-(k+1)})\chi_{I_k}(t).
\]
Using the lattice property and applying the \(p\)-upper estimate to
finite partial sums, and then passing to the limit by the Fatou
property of the function norm, we get
\[
\|f\|_X = \|f^*\|_X \preceq \left( \sum_{k=0}^{\infty} \left[
f^*(2^{-(k+1)})\|\chi_{I_k}\|_X \right]^p \right)^{1/p}.
\]
Since
\[
\|\chi_{I_k}\|_X \le \|\chi_{(0,2^{-k})}\|_X = \varphi_X(2^{-k}),
\]
and since \(\varphi_X\) is quasi-concave,
\[
\varphi_X(2^{-k}) \preceq \varphi_X(2^{-(k+1)}).
\]
Therefore
\[
\|f\|_X \preceq \left( \sum_{k=0}^{\infty} \left[
f^*(2^{-(k+1)})\varphi_X(2^{-(k+1)}) \right]^p \right)^{1/p}.
\]

We now compare this sum with the integral defining \(\Lambda^p(X)\).
If \(t\in I_{k+1}\), then \(t\le 2^{-(k+1)}\), and hence
\[
f^*(t)\ge f^*(2^{-(k+1)}).
\]
Moreover, by quasi-concavity of \(\varphi_X\),
\[
\varphi_X(t) \simeq \varphi_X(2^{-(k+1)}), \qquad t\in I_{k+1}.
\]
Consequently,
\[
\left[ f^*(2^{-(k+1)})\varphi_X(2^{-(k+1)}) \right]^p \preceq
\int_{I_{k+1}} \left(f^*(t)\varphi_X(t)\right)^p\,\frac{dt}{t}.
\]
Summing over \(k\ge0\), we obtain
\[
\sum_{k=0}^{\infty} \left[ f^*(2^{-(k+1)})\varphi_X(2^{-(k+1)})
\right]^p \preceq \int_0^1
\left(f^*(t)\varphi_X(t)\right)^p\,\frac{dt}{t}.
\]
Combining the previous estimates gives
\[
\|f\|_X \preceq \left( \int_0^1
\left(f^*(t)\varphi_X(t)\right)^p\,\frac{dt}{t} \right)^{1/p} =
\|f\|_{\Lambda^p(X)}.
\]
The proof is complete.
\end{proof}

\end{document}